\documentclass[11pt]{article}
\usepackage{amssymb, amsmath, hyperref, color}

\def\A{\mathbb{A}}
\def\C{\mathbb{C}}
\def\R{\mathbb{R}}
\def\N{\mathbb{N}}
\def\P{\mathbb{P}}
\def\Q{\mathbb{Q}}

\def\K{\mathbb{K}}

\def\rk {{\rm rank}}

\newtheorem{defn}{Definition}
\newtheorem{notn}[defn]{Notation}
\newtheorem{lemma}[defn]{Lemma}
\newtheorem{proposition}[defn]{Proposition}
\newtheorem{theorem}[defn]{Theorem}

\newtheorem{remark}[defn]{Remark}

\newtheorem{hypothesis}[defn]{Hypothesis}
\newtheorem{algor}[defn]{Algorithm}

           {\vspace{3.3mm}
           \noindent{\bf #1}\it}%
           {\vspace{3.3mm}}

\newenvironment{proof}[1]{
  \trivlist \item[\hskip \labelsep{\it #1}]}{\hfill\mbox{$\square$}
  \endtrivlist}

\def\ms {\medskip}

\def\cprime{$'$}

\title{On sign conditions over real multivariate polynomials}
\author{Gabriela Jeronimo$^{\textrm{a,\,c,}}$\thanks{Partially supported
by the following Argentinian research grants: UBACyT X847
(2006-2009), CONICET PIP 5852/05}
\thanks{Partially supported by ANPCYT PICT 2005 17--33018.} \quad Daniel
Perrucci$^{\textrm{a,}\ast}$ \quad Juan
Sabia$^{\textrm{b,\,c,}\ast}$
\\ \ \\
{$^{\textrm{a}}$ {\small Departamento de Matem\'atica, Facultad de
Ciencias
Exactas y Naturales,}}\\
{{\small Universidad de Buenos Aires, Ciudad Universitaria, 1428
Buenos Aires, Argentina}}
\\[2mm]
{$^{\textrm{b}}$ {\small Departamento de Ciencias Exactas, Ciclo B\'asico Com\'un,}}\\
{{\small Universidad de Buenos Aires, Ciudad Universitaria, 1428 Buenos Aires, Argentina}}\\[2mm]
{$^{\textrm{c}}$ {\small CONICET - Argentina }}
 }

\begin{document}
\maketitle

\begin{abstract}
We present a new probabilistic algorithm to find a finite set of
points intersecting the closure of each connected component of the
realization of every sign condition over a family of real
polynomials defining regular hypersurfaces that intersect
transversally. This enables us to show a probabilistic procedure to
list all feasible sign conditions over the polynomials. In addition,
we extend these results to the case of closed sign conditions over
an \emph{arbitrary} family of real multivariate polynomials.
 The complexity bounds for
these procedures improve the known ones.
\end{abstract}


\section{Introduction}

Given polynomials $f_1,\dots, f_m \in \R[x_1,\dots, x_n]$, a sign
condition $\sigma \in \{<,=,>\}^m$, or a closed sign condition
$\sigma \in \{\le,=,\ge\}^m$, is said to be \emph{feasible} if the
system $f_1(x) \sigma_1 0,\dots, f_m(x)\sigma_m 0$ has a solution in
$\R^n$, and the set of its solutions is called the {\it realization}
of $\sigma$. One of the basic problems in computational
semialgebraic geometry is to decide whether a sign condition is
feasible. This problem is a particular case of quantifier
elimination and, on the other hand, many elimination algorithms use
subroutines determining all the feasible sign conditions for a
family of polynomials.

The first elimination algorithms over the reals are due to Tarski
\cite{Tarski51} and Seidenberg \cite{Seidenberg54}, but their
complexities are not elementary recursive. Collins \cite{Collins75}
was the first to obtain a doubly exponential complexity. In
\cite{GriVo88}, Grigor{\cprime}ev and Vorobjov present  an algorithm
with single exponential complexity to decide the consistency of a
system of equalities and inequalities by studying the critical
points of a function in order to obtain a finite set of points
intersecting each connected component of the solution set. This same
idea was used to obtain more efficient quantifier elimination
procedures (see \cite{HeRoSo90}, \cite{Renegar92} and
\cite{BaPoRo96}). A standard technique is to take sums of squares
and introduce infinitesimals to reduce the problem to the study of a
smooth and compact real hypersurface. The specific problem of
consistency for equalities over $\R$ was also treated through the
critical point method afterwards. In \cite{RRS} the non-emptiness of
a real variety defined by a single equation is studied, reducing the
introduction of infinitesimals and, in \cite{AuRoSED02}, an
algorithm with no infinitesimals is given to deal with arbitrary
positive dimensional systems.

Several probabilistic procedures lead to successive complexity
improvements. Using classical polar varieties, in \cite{BGHM97} and
\cite{BGHM01}, the case of a smooth compact variety given by a
regular sequence is tackled within a complexity depending
polynomially on an intrinsic degree of the systems involved and the
input length. To achieve this complexity, straight-line program
encoding of polynomials  and an efficient procedure to solve
polynomial equation systems over the complex numbers (\cite{GHHM+})
are used. The compactness assumption is dropped in \cite{BGHP04} and
\cite{BGHP05}, by introducing  generalized polar varieties. The
non-compact case is also considered in \cite{SafSch03} for a smooth
equidimensional variety defined by a radical ideal by studying
projections over polar varieties, and an extension to the
non-equidimensional situation is presented in \cite{SafTre}.
Finally, \cite{Safey08} deals with sets of the type $\{f>0\}$
through the computation of generalized critical points.

In this paper, we consider the problem of determining all feasible
sign conditions (or closed sign conditions) over a given finite
family of multivariate polynomials. We first present a probabilistic
algorithm that, under certain regularity conditions (see Hypothesis
\ref{la_hipotesis}), obtains a finite set of points intersecting the
closure of each connected component of the realization of every sign
condition over the given polynomials. For families of
\emph{arbitrary} polynomials, we show a probabilistic algorithm that
computes a finite set of points intersecting each connected
component of the realization of every \emph{closed} sign condition.
The input and intermediate computations in our algorithms are
encoded by straight-line programs (see Section \ref{algycomp}). The
output is described by means of geometric resolutions, that is to
say, by univariate rational parametrizations of $0$-dimensional
varieties. In both situations, the output of the algorithm enables
us to determine all the feasible closed sign conditions over the
polynomials by evaluating their signs at the computed points, which
is done by using the techniques in \cite{Canny93}. Moreover, in the
first case, we can determine all feasible sign conditions (for a
treatment of the problem in full generality, see for instance
\cite{BPR93}).

A sketch of our main algorithms is the following. Given $f_1,\dots,
f_m \in \R [x_1,\dots, x_n]$, a generic change of variables prevents
asymptotic behavior  with respect to the projection to the first
coordinate $x_1$ for each connected component $C\subset \R^n$ of
every feasible (closed) sign condition over $f_1,\dots, f_m$: either
$C$ projects onto  $\R$ or its projection is a proper interval whose
endpoints have a nonempty finite fiber in $\overline C$. In the
latter case, points in $\overline C$ are obtained as extremal points
of $x_1$. These extremal points are solutions of particular systems
of polynomial equations which are dealt with by deformation
techniques that enable the computation of geometric resolutions of
finite sets including them. To find points in the components
projecting onto $\R$, the set is intersected with $\{x_1 = p_1\}$
for a particular value $p_1$, and the algorithm continues
recursively.

The following theorem states our main results (see Theorems
\ref{algoritmopuntosencomponentes} and
\ref{algoritmopuntosencomponentescerradas}):

\medskip

\noindent \textbf{Theorem } \emph{Let $\K$ be an effective subfield
of $\R$. There are probabilistic algorithms to perform the following
tasks:
\begin{itemize}
\item Given polynomials $f_1, \dots, f_m \in \K[x_1, \dots,
x_n]$ satisfying Hypothesis \ref{la_hipotesis}, with degrees bounded
by $d\ge 2$ and encoded by a straight-line program of length $L$,
obtain a finite set of points intersecting the closure of each
connected component of the realization of every sign condition over
$f_1,\dots, f_m$ within $O\big(\sum_{s = 1}^{\min\{m, n\}}
\binom{m}{s}\big(\binom{n-1}{s-1}d^{n}\big)^2   L \big)$ operations
in $\K$ up to poly-logarithmic factors.
\item Given \emph{arbitrary} polynomials $f_1, \dots, f_m \in \K[x_1, \dots,
x_n]$, with degrees bounded by an even integer $d$ and encoded by a
straight-line program of length $L$, obtain a finite set of points
intersecting each connected component of the realization of every
closed sign condition over $f_1,\dots, f_m$ within  $O\big(\sum_{s =
1}^{\min\{m, n\}} 2^s\binom{m}{s}\big(\binom{n-1}{s-1}d^{n}\big)^2
(L+d) \big)$ operations in $\K$ up to poly-logarithmic factors.
\end{itemize}}

\medskip

The factor $\binom{n-1}{s-1}d^{n}$ in the complexity estimates is an
upper bound for the bihomogeneous B\'ezout numbers arising from the
Lagrange characterization of critical points of projections
(cf.~\cite{SafTre}). In fact, one of the new tools to achieve the
stated complexity order, which improves the previous ones depending
on the same parameters, is the use of algorithmic deformation
techniques specially designed for bihomogeneous systems (for a
similar approach, see \cite{HJSS}). Up to now, the polynomial
systems used to characterize critical points were handled with
general algorithms solving polynomial equations over the complex
numbers (see, for instance, \cite{ABRW96}, \cite{Rouillier99},
\cite{GLS01} and \cite{Lecerf03}). As taking sums of squares of
polynomials and introducing infinitesimals lead to an artificial
growth of the parameters involved in the complexity estimates,
another important feature of our techniques is that we work directly
with the input equations instead of using these constructions.

This work can be seen as an extension of \cite{SafSch03} and of
\cite{BGHP05} in the sense that we deal not only with equations but
also with inequalities. In particular, the algorithm  in
\cite{SafSch03}, which only works for the  case of smooth
equidimensional varieties defined by a radical ideal, considers a
family of equation systems equivalent to the ones introduced in the
recursive stages of our algorithm. However, those systems involve a
large number of polynomials and do not have any evident structure.

We also prove that our deformation based approach can be applied to
deal with sign conditions over bivariate systems without any
assumption on the polynomials. We expect this can be extended to
general multivariate polynomials. This is the subject of our current
research. Finally, we adapt our techniques to the case of an
arbitrary multivariate polynomial.

This paper is organized as follows: In Section \ref{preliminares} we
introduce some basic notions and notation that will be used
throughout the paper. Section \ref{preparacion} is devoted to
presenting the basic ingredients to be used in the design of our
algorithms. In Section \ref{seccion_hipotesis}, we present our main
algorithms to determine all feasible sign conditions over polynomial
families satisfying regularity assumptions. In Section
\ref{sec:closed} we consider the same problem for closed sign
conditions over arbitrary multivariate polynomials. The last section
contains our results on sign conditions over bivariate polynomial
families and over a single multivariate polynomial.

\section{Preliminaries}\label{preliminares}

\subsection{Notation}\label{notation}

Throughout this paper $\mathbb{Q}$, $\mathbb{R}$ and $\mathbb{C}$
denote the fields of rational, real and complex numbers
respectively, $\mathbb{N}$ denotes the set of positive integers and
$\mathbb{N}_0:= \mathbb{N} \cup \{ 0 \}$. If $k$ is a field, $\bar
k$ will denote an algebraic closure of $k$.

For $n\in \mathbb{N}$ and an algebraically closed field $k$, we
denote by $\mathbb{A}^n_k$ and $\mathbb{P}^n_k$ (or simply by
$\mathbb{A}^n$ or $\mathbb{P}^n$ if the base field is clear from the
context) the $n$-dimensional affine space and projective space over
$k$ respectively, equipped with their Zariski topologies. For a
subset $X$ of one of these spaces, we will denote by $\overline X$
its closure. We adopt the usual notions of dimension and degree of
an algebraic variety $V$ (see for instance \cite{Shafarevich} and
\cite{Heintz83}).

We will denote projections on certain coordinates $x$ by $\pi_x$.
For short, a projection on the $k$th coordinate will also be denoted
by $\pi_k$.

For any non-empty set $A \subset \R^n$, $\overline A$ will denote
its closure with respect to the usual Euclidean topology. We define
$Z_{\inf}(A) = \{(x_1, \dots, x_n) \in \overline{A} \mid x_1 = \inf
\pi_1(A)\}$ if $\pi_1(A)$ is bounded from below, and $Z_{\inf}(A) =
\emptyset$ otherwise. Similarly, $Z_{\sup}(A) = \{(x_1, \dots, x_n)
\in \overline{A} \mid x_1 = \sup \pi_1(A)\}$ if $\pi_1(A)$ is
bounded from above, and $Z_{\sup}(A)= \emptyset$ otherwise. Finally,
we denote $Z(A)  = Z_{\inf}(A) \cup Z_{\sup}(A)$.

Throughout this paper, $\log$ will denote logarithm to the base $2$.

\subsection{Algorithms and complexity}\label{algycomp}

The algorithms we consider in this paper are described by arithmetic
networks over an effective base field $\mathbb{K}\subset \R$ (see
\cite{vzg86}). The  notion of {\em complexity} of an algorithm we
consider is the number of operations and comparisons over
$\mathbb{K}$.

The objects we deal with are polynomials with coefficients in
$\mathbb{K}$. Throughout our algorithms we represent each polynomial
either as the array of all its coefficients in a pre-fixed order of
its monomials ({\it dense form}) or by a {\it straight-line
program}. Roughly speaking, a straight-line program (or slp, for
short) over $\mathbb{K}$ encoding a list of polynomials in
$\mathbb{K}[x_1,\dots, x_n]$ is a program without branches (an
arithmetic circuit) which enables us to evaluate these polynomials
at any given point in $\mathbb{K}^n$. The number of instructions in
the program is called the {\em length} of the slp (for a precise
definition  we refer to \cite[Definition 4.2]{Burgisser}; see also
\cite{HS82}).

We will do operations with polynomials encoded in both these ways.
To estimate the complexities we will use the following results:
Operations  between univariate polynomials with coefficients in a
field $\K$ of degree bounded by $d$ in dense form can be done using
$O(d \log (d) \log \log (d))$ operations in $\K$ (see \cite[Chapter
8]{vzG}). From an slp of length $L$ encoding a polynomial $f \in
\K[x_1, \dots, x_n]$, we can compute an slp with length $O(L)$
encoding $f$ and all its first order partial derivatives (see
\cite{BS83}).

\subsection{Geometric resolutions}\label{geometricresolutions}

A way of representing zero-dimensional affine varieties which is
widely used in computer algebra nowadays is  a \emph{geometric
resolution}. This notion was first introduced in the works of
Kronecker and K{\"o}nig in the last years of the XIXth century
(\cite{Kronecker} and \cite{Koenig}) and appears in the literature
under different names (rational univariate representation, shape
lemma, etc.). For a detailed historical account on its application
in the algorithmic framework, we refer the reader to \cite{GLS01}.
The precise definition we are going to use is the following:

Let $k$ be a field of characteristic $0$ and $V = \{ \xi^{(1)},
\dots,\xi^{(D)} \} \subset \A_{\bar k}^n$ be a zero-dimensional
variety defined by polynomials in $k[x_1, \dots, x_n]$. Given a
\emph{separating} linear form $\ell = u_1 x_1 + \dots + u_n x_n\in
k[x_1, \dots, x_n]$ for $V$ (that is, a linear form $\ell$ such that
$\ell (\xi^{(i )})\ne \ell(\xi^{(j)})$ if $i \ne j$), the following
polynomials completely characterize the variety $V$:
\begin{itemize}
\item the \emph{minimal polynomial} $q:= \prod_{1 \le i \le D} (U
- \ell(\xi^{(i)}))\in k[U]$ of $\ell$ over the variety $V$ (where
$U$ is a new variable), \item a polynomial $\tilde q\in k[U]$ with
$\deg(\tilde q) <D$ and relatively prime to $q$, \item polynomials
$w_1,\dots, w_n \in k[U]$ with $\deg( w_j )< D$ for every $1\le j
\le n$ satisfying
$$ V = \big\{ \big(\frac{w_1}{\tilde q}(\eta) ,\dots, \frac{w_n}{\tilde q} (\eta) \big) \in
\overline{k}^n \mid \eta \in \overline{k} ,\ q(\eta) = 0 \big\}.$$
\end{itemize}
The family of univariate polynomials $q, \tilde q, w_1,\dots, w_n\in
k[U]$ is called a \emph{geometric resolution} of $V$ (associated
with the linear form $\ell$).

We point out that the polynomial $\tilde q$ appearing in the above
definition is invertible in $k[U]/(q(U))$. Setting $v_k(U):= \tilde
q^{-1}(U) w_k(U) \mod (q(U))$ for every $1\le k \le n$, we are lead
to the standard notion of geometric resolution: a family of $n+1$
polynomials $q, v_1, \dots, v_n$ in $k[U]$  satisfying $V = \big\{
\big(v_1(\eta) ,\dots, v_n(\eta) \big) \in \overline{k}^n \mid \eta
\in \overline{k} ,\ q(\eta) = 0 \big\}$. We will use both
definitions alternatively, since the complexity of passing from one
representation to the other does not modify the overall complexity
of our algorithms. Which notion is used in each case will be clear
from the number of polynomials.

\section{General approach}\label{preparacion}

\subsection{Avoiding asymptotic situations}\label{cambios_de_variable_genericos}

  For  any non-empty set $A \subset
\R^n$ we define $Z_{\inf}(A,k) =  \{(x_1, \dots, x_n) \in
\overline{A} \ | \ x_{k} = \inf \pi_k(A)\}$ if $A$ is bounded from
below and $Z_{\inf}(A,k) = \emptyset$ otherwise. Similarly,
$Z_{\sup}(A,k) = \{ (x_1, \dots, x_n) \in \overline{A} \ | \ x_{k} =
\sup \pi_k(A)\}$ whenever $A$ is bounded from above and
$Z_{\sup}(A,k) = \emptyset$ otherwise. Finally, $Z(A,k)=
Z_{\inf}(A,k) \cup Z_{\sup}(A,k)$. In particular, when $k =1$,
$Z_{\inf}(A, 1)$, $Z_{\sup}(A, 1)$ and $Z(A, 1)$ will be denoted by
$Z_{\inf}(A)$, $Z_{\sup}(A)$ and $Z(A)$ respectively as has already
been stated in Section \ref{notation}.

The precise conditions achieved by a generic linear change of
variables are stated in the following:

\begin{proposition} \label{cambio_var_gen_finit_extrem}
Let $f_1, \dots, f_m $ be $n$-variate polynomials with real
coefficients. After a generic linear change of variables over $\Q$,
for every  semialgebraic set $\mathcal{P}$ defined in $\R^n$ by a
Boolean formula on the polynomials $f_1,\dots, f_m$  involving
equalities and inequalities to zero and every $p=(p_1, \dots,
p_{n})$ $\in \R^{n}$, if $1 \le k \le n$ and $C$ is a connected
component of $\mathcal{P} \cap \{x_1 = p_1, \dots,
x_{k-1}=p_{k-1}\}$, then $Z(C,k)$ is a finite set (possibly empty).
Moreover, if $\pi_k(C) $ is bounded from below, then $Z_{\inf}(C,k)$
is not empty, and, if $\pi_k(C)$ is bounded from above, then
$Z_{\sup}(C,k)$ is not empty.
\end{proposition}

To prove Proposition \ref{cambio_var_gen_finit_extrem}, we will use
the following auxiliary lemma:

\begin{lemma} \label{fin_punt_mas_izq}
Let $\{f_{ij}\}_{1 \le i \le n, 1 \le j \le l_i} \subset \R[x_1,
\dots, x_n]$ be a family of nonzero polynomials satisfying
simultaneously:
\begin{enumerate}

\item[a)] for $1 \le i \le n$, $\{f_{ij}\}_{1 \le j \le l_i}$ is
contained in $\R[x_1, \dots, x_i]$, it is closed under derivation
with respect to the variable $x_i$, and every polynomial in it is
quasi-monic (that is, monic up to a constant) with respect to $x_i$,

\item[b)] for $1 < i \le n$, $\{f_{(i-1)j}\}_{1 \le j \le
l_{i-1}}$ slices $\{f_{ij}\}_{1 \le j \le l_i}$ in the sense of
\cite[Definition 2.3.4]{BCR}.
\end{enumerate}
Let $p = (p_1, \dots, p_{n}) \in \R^{n}$, $1 \le i \le n$, and
$\mathcal{P}\subset \R^i$ be a semialgebraic set defined by a
Boolean formula on the polynomials $f_{ij}$, $1\le j \le l_i$,
involving equalities and inequalities to zero. For  $1\le k \le i$,
let $C$ be a connected component of $\mathcal{P} \cap \{x_1 = p_1,
\dots, x_{k-1}=p_{k-1}\}$. Then the set $Z(C,k)$ is finite (possibly
empty). Moreover, if $\pi_k(C) $ is bounded from below, then
$Z_{\inf}(C,k) \ne \emptyset,$ and, if $\pi_k(C)$ is bounded from
above, then $Z_{\sup}(C,k)\ne \emptyset.$
\end{lemma}

\begin{proof}{Proof: } As for every
$1 \le k \le n$ the family $\{f_{ij}(p_1, \dots, p_{k-1}, x_k,
\dots, x_n)\}_{k \le i \le n, 1 \le j \le l_i}$ $\subset \R[x_k,
\dots, x_n]$ satisfies the hypotheses,
 it is enough to prove the lemma for  $k = 1$.

For $i = 1$, the result is clear.

Suppose the statement is true for $i-1$. Let $\pi: \R^i \to
\R^{i-1}$ be the projection on the first $i-1$ coordinates.
Following the notation in  \cite[Chapter 2]{BCR}, let $A_1, \dots,
A_\ell$ be the semialgebraic sets giving the slicing of $\R^{i-1}$
with respect to $f_{i1}, \dots, f_{il_i}$ given by the polynomials
$f_{(i-1)1}, \dots, f_{(i-1)l_{i-1}}$ and, for $1 \le s \le \ell$,
let $\xi_{s,1} < \dots < \xi_{s, a_s}:A_s \to \R$ be the continuous
semialgebraic functions that slice $A_s \times \R$. Let $A_{s,1},
\dots, A_{s, u_s}$ be the connected components of  $A_s$.

Note that $C$ is a finite union of some sets of the partitions of
the sets $A_{s,u} \times \R$ given by  $\xi_{s,1}, \dots, \xi_{s,
a_s}$ and  $\pi(C)$ is a finite union of sets $A_{s, u}$. If $\pi(C)
= \cup_h A_{s_h, u_h}$, then $Z(\pi(C)) \subset \cup_h Z(A_{s_h,
u_h})$. Since each $A_{s_h, u_h}$ is a connected component of
$A_{s_h}$, which can be described by a Boolean formula involving
equalities and inequalities to zero of  $f_{(i-1)1}, \dots,
f_{(i-1)l_{i-1}}$,  by inductive hypothesis,  each $Z(A_{s_h, u_h})$
is finite and, therefore,  $Z(\pi(C))$ is finite too.

Let $w  \in Z(C)$. As $Z(C) \subset \overline{C}$, $\pi(w) \in
\pi(\overline{C}) \subset \overline{\pi(C)}$. Moreover, as $\pi_1(C)
= \pi_1(\pi(C))$,   $\pi(w) \in Z(\pi(C))$. On the other hand, as $w
\in Z(C)$, at least one of the quasi-monic polynomials $f_{i1},
\dots, f_{il_i}$ vanishes at $w$ (otherwise, all the $f_{ij}, j= 1,
\dots, l_i,$ would have a constant sign in a neighborhood of $w$
contained in $C$). Therefore,  $Z(C)$ is a finite set since $\pi(w)$
lies in the finite set $Z(\pi(C))$ and $w_i$ is a zero of one of the
nonzero polynomials $f_{i1}(\pi(w), x_i), \dots, f_{il_i}(\pi(w),
x_i)$.

Suppose now that $\pi_1(C)$ is an interval bounded, for example,
from below. Then, by the inductive assumption, there exists $z =
(z_1, \dots, z_{i-1}) \in Z_{\inf}(\pi(C)) \subset
\overline{\pi(C)}$. Assume further that $A_{1,1} \subset \pi(C)$, $z
\in Z_{\inf}(A_{1,1})$ and $\gamma:[0,1] \to \overline{A}_{1,1}$ is
a continuous semialgebraic curve such that $\gamma((0,1]) \subset
A_{1,1}$ and $\gamma(0) = z$ (see \cite[Theorem 2.5.5]{BCR}). Let
$\tilde x=\gamma(1)$. Since $\tilde x \in A_{1,1} \subset \pi(C)$,
there exists $y \in \R$ such that $(\tilde x, y) \in C$. Using
\cite[Lema 2.5.6]{BCR}, each $\xi_{1,a}$ can be extended
continuously to $\overline{A}_1$. Let us denote by $\xi_{1,a}$ also
this extension. Depending on the position of $y$ with respect to the
values $\xi_{1,1}(\tilde x) < \dots < \xi_{1,a_1}(\tilde x)$, it is
easy in any case to define a continuous semialgebraic function
$h:[0, 1] \to \R$ such that the continuous function
$\tilde{\gamma}:[0, 1]\to \R^i$ defined as $\tilde{\gamma}(t) =
(\gamma(t), h(t))$  satisfies $\tilde{\gamma}((0, 1]) \subset C$
(note that the signs of the polynomials $f_{i1}, \dots, f_{il_i}$
are constant over $\tilde{\gamma}((0, 1]))$ and, therefore,  $(z,
h(0)) = \tilde{\gamma}(0) \in \overline{C}$. Moreover, as $z_1 =
\inf \pi_1(\pi(C)) = \inf \pi_1(C)$,  $(z, h(0)) \in Z_{\inf}(C)$.
\end{proof}

Now, we can prove Proposition \ref{cambio_var_gen_finit_extrem}:

\begin{proof}{Proof: } By Lemma \ref{fin_punt_mas_izq},
it suffices to show that there exists a Zarisky open set  ${\cal U}
\subset Gl(n, \C)$ such that, for every $V_0 \in \Q^{n \times n}
\cap {\cal U}$ there exists a family of polynomials $\{f_{ij}\}_{1
\le i \le n, 1 \le j \le l_i} \subset \R[x]$ satisfying the
hypotheses of the lemma, and such that, for  $1 \le j \le m$,
$f_{nj}(x) = f_j(V_0x)$ with $m \le l_n$. Let $V$ be a matrix whose
entries are new variables
 $v_{rs}$, $1 \le r, s
\le n$ and consider  $\{F_{ij}\}_{1 \le i \le n, 1 \le j \le l_i}
\subset \R[v, x]$ defined in the following way:
\begin{itemize}
\item Take $l'_n = m$ and, for $1 \le j \le l'_n$, let
$F_{nj}(V, x) = f_j(Vx)$. Then, for $1 \le j_0 \le l'_n$, if $\deg_x
F_{nj_0} = d_{nj_0}$, add the first $d_{nj_0}-1$ derivatives of
$F_{nj_0}$ with respect to  $x_n$ to the list to obtain
$\{F_{nj}\}_{1\le j \le l_n}$.
\item From  $\{F_{(i_0+1)j}\}_{1 \le j
\le l_{i_0+1}} \subset \R[v, x_1, \dots, x_{i_0+1}]$, first, form
$\{F_{i_0j}\}_{1 \le j \le l'_{i_0}} \subset \R[v, x_1, \dots,
x_{i_0}]$ by taking all possible resultants and subresultants with
respect to the variable $x_{i_0 + 1}$ between pairs of polynomials,
not taking into account the ones that are identically zero. Then,
for $1 \le j_0 \le l'_{i_0}$, if $\deg_x F_{i_0j_0} = d_{i_0j_0}$,
add the first $d_{i_0j_0}-1$ derivatives of $F_{i_0j_0}$ with
respect to the variable $x_{i_0}$ to obtain the family
$\{F_{i_0j}\}_{1\le j \le l_{i_0}}$.
\end{itemize}

Let $1 \le i \le n$ and $1 \le j \le l_i$. Let $d_{ij}:= \deg_x
F_{ij}$  and let $q_{ij}\in \R[v]$ be the coefficient of the
monomial $x_{i}^{d_{ij}}$ in $F_{ij}\in \R[v][x]$. We will show that
$q_{ij}\ne 0 $. To do so, we also prove the following: for $A \in
\Q^{i \times i}$, denote by
$\tilde A$ and $\hat A$ the matrices $ \tilde A = \left(%
\begin{array}{cc}
  A & 0 \\
  0 & 1 \\
\end{array}%
\right) \in \Q^{(i+1) \times (i+1)}$ and $ \hat A = \left(%
\begin{array}{cc}
  A & 0 \\
  0 & Id_{n-i} \\
\end{array}
\right) \in \Q^{n \times n}$; then, for every  $1 \le j \le l'_i$
and every $B \in \Q^{i \times i}$, $F_{ij}(V, Bx) = F_{ij}(V\hat
B,x)$, and, for every $l'_{i}+1 \le j \le l_i$ and  $B \in \Q^{(i-1)
\times (i-1)}$, $F_{ij}(V, \tilde Bx) = F_{ij}(V\hat B,x)$.

 If $i = n$, it is clear that $q_{nj} \ne 0$  for every $1\le j\le l'_n$ because there
 is
 a change of variables $V_0 \in \Q^{n \times n}$
which makes the polynomials $f_{j}$, for $1 \le j \le m$,
quasi-monic of degree $d_{nj}$ with respect to the variable $x_n$.
Moreover, for $1 \le j \le l'_n$ and for $B \in \Q^{n \times n}$,
$F_{nj}(V, Bx) = f_j(VBx) = F_{nj}(VB,x)$. For $l'_n + 1\le j \le
l_n$ and $B \in \Q^{(n-1) \times (n-1)}$, suppose $F_{nj}$ is the
derivative with respect to $x_n$ of $F_{nj'}(V, x)$.  Then, by
differentiating  $F_{nj'}(V, \tilde Bx) = F_{nj'}(V\hat B,x)$ with
respect to $x_n$, we obtain the desired identity.

Let  $i < n$ and $B \in \Q^{i\times i}$. Note that, if  $1 \le j \le
l_{i+1}$ and $$F_{(i+1)j}(V, x) = \sum_{h =
0}^{d_{(i+1)j}}q_{(i+1)j,h}(V, x_1, \dots, x_i)x_{i+1}^h,$$ as
$F_{(i+1)j}(V, \tilde Bx) = F_{(i+1)j}(V\hat B, x)$, then
$q_{(i+1)j,h}(V, Bx) = q_{(i+1)j,h}(V\hat B ,x)$ for $1 \le h\le
d_{(i+1)j}$. Let $1 \le j \le l'_i$, and suppose  $F_{ij}$ is a
resultant or a subresultant between $F_{(i+1)1}$ and $F_{(i+1)2}$.
As $\deg_{x_{i+1}}F_{(i+1)1} = d_{(i+1)1}$ and
$\deg_{x_{i+1}}F_{(i+1)2} =d_{(i+1)2} $, $F_{ij}$ corresponds to the
evaluation of a certain polynomial $R$ in $q_{(i+1)1,h}$ and
$q_{(i+1)2,h}$. Then $F_{ij}(V, Bx) = R(q_{(i+1)1,1}(V, Bx), \dots,
q_{(i+1)2,d_{(i+1)2}}(V, Bx)) = R(q_{(i+1)1,1}(V \hat B, x), \dots,
q_{(i+1)2,d_{(i+1)2}}(V \hat B, x)) = F_{ij}(V \hat B, x).$ When
$l'_i+1 \le j \le l_i$, we may proceed in the same way as we do for
$i = n$. Now, it suffices to prove that $q_{ij}\ne 0 $ for $1 \le j
\le l'_{i}$. Let $B_0 \in \Q^{i \times i}$ such that $\deg
_{x_i}F_{ij}(V, B_0x) =d_{ij}$. As $F_{ij}(V, B_0x) = F_{ij}(V\hat
B_0, x)$, then $q_{ij}(V\hat B_0)\ne 0$ and, therefore $q_{ij}\ne
0$.

Define ${\cal U} = \{V_0 \in \C^{n \times n} \ | \ q_{ij}(V_0) \ne 0
\hbox{ for } 1 \le i \le n, 1 \le j \le l_i\}$. By \cite[Proposition
4.34 and Theorem 5.14]{BPR}, for every $V_0 \in \Q^{n \times n} \cap
{\cal U}$, the set $\{f_{ij}(x)\}_{1 \le i \le n, 1 \le j \le l_i}$
defined by $f_{ij}(x) = F_{ij}(V_0, x)$ satisfies both conditions in
Lemma \ref{fin_punt_mas_izq}.
\end{proof}

The following proposition is a major tool for our algorithms (cf.
\cite[Theorem 2]{SafSch03}).

\begin{proposition}\label{idea_algoritmo} Let $f_1,\dots, f_m$ be $n$-variate polynomials with real coefficients.
After a generic change of variables, for every semialgebraic
 set $\mathcal{P}$  defined in $\R^n$ by a Boolean
formula on $f_1,\dots, f_m$ involving equalities and inequalities to
zero, and every $p=(p_1, \dots, p_{n})\in \R^{n}$, if for $1 \le k
\le n$, $\mathcal{P}(k,p)$  is the set of all the connected
components of $\mathcal{P}\cap \{x_1 = p_1, \dots, x_{k-1} =
p_{k-1}\}$, then
$$
\{p\} \cup \Big( \bigcup_{k=1}^{n} \bigcup_{C \in \mathcal{P}(k,p)}
Z(C,k) \Big)
$$
is finite and intersects the closure of each connected component of
$\mathcal{P}$.
\end{proposition}

\begin{proof}{Proof: } Proposition \ref{cambio_var_gen_finit_extrem}
ensures that this set is finite. Let $C_1 \in \mathcal{P}(1,p)$. If
$\pi_1(C_1)$ is bounded from above or below, again Proposition
\ref{cambio_var_gen_finit_extrem} states that $Z(C_1,1)$ is a finite
non-empty set and is included in $\overline{C}_1$. Otherwise,  $
\pi_1(C_1) = \R$ and $C_1 \cap \{x_1 = p_1\} \ne \emptyset$. Let
$C_2 \in \mathcal{P}(2, p)$ be a connected component of $C_1 \cap
\{x_1 = p_1\}$. If $\pi_2(C_2)$ is bounded from above or below,
$Z(C_2, 2)\ne \emptyset$ and  is included in $\overline{C}_2\subset
\overline{C}_1$. Otherwise, $\pi_2(C_2) = \R$ and $C_1 \cap \{x_1 =
p_1, x_2 = p_2\} \ne \emptyset$. Following this procedure, we obtain
that either there exists $C_k \in \mathcal{P}(k,p)$  such that
$Z(C_k,k)\ne \emptyset$ and  is included in $\overline{C}_k\subset
\overline{C}_1$ for some $1 \le k \le n$ or $p \in C_1$.
\end{proof}

The  proof above leads to the recursive structure of our algorithm:
For $2\le k \le n$, we may think the $k$-th  variable as the first
one for the polynomials $f_j(p_1, \dots, p_{k-1}, x_k, \dots, x_n)$
for $1\le j \le  m$. Therefore, it is enough to consider the problem
of finding extremal points for the projection over the first
coordinate of the closures of the connected components of a
semialgebraic set.

\subsection{Equations defining extremal points}
\label{definicion_sistemas}

Let $f_{1} , \dots , f_{m} \in \R[x_1, \dots, x_n]$ and let $ S:= \{
i_1,\dots, i_s\} \subset \{1, \dots, m\}$. If $1\le s \le n-1$, the
implicit function theorem implies that a point $z$ with maximum or
minimum first coordinate in a connected component of $\{f_{i_1} =
\dots = f_{i_s} = 0\}$ satisfies \begin{equation}\label{rankift}
f_{i_1}(z)= \dots = f_{i_s}(z) = 0, \ \rk \left(
\begin{array} {ccc}
\frac{\partial f_{i_1}}{\partial x_2} (z) & \cdots & \frac{\partial f_{i_1}}{\partial x_n} (z)  \\
\vdots & \ddots & \vdots \\
\frac{\partial f_{i_s}}{\partial x_2} (z) & \cdots & \frac{\partial
f_{i_s}}{\partial x_n}(z)
\end{array}
\right) < s.\end{equation} This condition can be rewritten as
\begin{equation}\label{J_S}
\left\{\begin{array}{l}f_{i_1}(z)= \dots = f_{i_s}(z) = 0,
\\[1mm]
\sum_{j = 1}^s \mu_j\overline\nabla f_{i_j}(z) = (0, \dots,
0).\end{array}\right.
\end{equation}
for  $\mu_1, \dots, \mu_s \in \R$ not simultaneously zero, where
$\overline\nabla f_{i_j}(z)$ denotes the vector obtained by removing
the first coordinate from the gradient $\nabla f_{i_j}(z)$.

When $s=1$, for $z \in\R^n$, conditions (\ref{rankift}) are
equivalent to
\begin{equation} \label{cond_1}
f_{i_1}(z)= \frac{\partial f_{i_1}}{\partial x_2}(z) = \dots =
\frac{\partial f_{i_1}}{\partial x_n}(z) = 0.
\end{equation}
When $s\ge n$, we will simply consider the conditions
\begin{equation} \label{cond_s_grande}
f_{i_1}(z)=  \dots = f_{i_s}(z)= 0.
\end{equation}

For every $2\le s \le n-1$, as system (\ref{J_S}) is homogeneous in
the variables $\mu_1, \dots, \mu_s$, we consider the variety $W_S
\subset \A^n_{\C} \times \P_{\C}^{s-1}$ defined as the zero set of
this system. If $s=1$ or $s\ge n$, let $W_S$ be the variety defined
by systems (\ref{cond_1}) and (\ref{cond_s_grande}) respectively.

The following result is an adaptation to our context of the
Karush-Kuhn-Tucker conditions (see \cite{Ped04}) from non-linear
optimization which generalize the Lagrange multipliers theorem in
order to consider equality and inequality constraints.

\begin{proposition} \label{karush_kuhn_tucker} Let $f_1,\dots, f_m \in \R[x_1,\dots, x_n]$ and
$\sigma=(\sigma_1,\dots, \sigma_m)\in \{\le, <, = , > , \ge\}^m$.
Set $E_\sigma=\{ i \mid \sigma_i = ``="\}$. Then, for every
connected component $C$ of the set $\{ x\in \R^n \mid f_1(x)
\sigma_1 0, \dots, f_m(x) \sigma_m 0 \}$, we have
$$Z(C) \subset  \bigcup_{E_\sigma \subset S \subset \{1, \dots, m\},\atop{S \ne \emptyset}} \pi_x(W_S)
.$$
\end{proposition}
\begin{proof}{Proof: } Without loss of generality, assume  $E_\sigma = \{ 1,\dots, l\}$
(or $E_\sigma = \emptyset$ and $l=0$) and $\pi_1(C)$ is bounded from
below. Let $z=(z_1, \dots, z_n) \in Z_{\inf}(C)$ and $S_0 = \{i \in
\{1, \dots, m\} \mid f_i(z) = 0\}$. Note that $E_\sigma \subset S_0$
and, even when $l = 0$, $S_0\ne \emptyset$; then, we may assume that
$S_0=\{1, \dots, t\}$ with $\max\{1,l\}\le t \le m$. We will show
that $z\in \pi_x(W_{S_0})$.

If $t\ge n$, we have  $z\in \pi_x(W_{S_0})$ by the definition of
this set.

Assume now that $t\le n-1$. If  $z\notin \pi_x(W_{S_0})$, the set $
\{\overline \nabla f_i(z), i\in S_0\}$ is linearly independent. Let
$f: \R^n \to \R^t$ be the map $f = (f_1, \dots, f_t)$. We may assume
that the minor corresponding to the variables $n-t+1, \dots, n$ in
the Jacobian matrix $Df(z)$ is not zero. Applying the inverse
function theorem to $h(x) = (x_1 - z_1, \dots, x_{n-t} - z_{n-t},
f_1(x), \dots, f_t(x))$, there exist an open neighborhood $U$ of
$z$, $\varepsilon \in \R_{> 0}$ and a map $g:(-\varepsilon,
\varepsilon)^n \to U$ inverse to $h: U \to (-\varepsilon,
\varepsilon)^n$. Moreover, we may assume that $f_{t+1},\dots, f_m$
have constant signs on $U$.

Let $w \in C \cap U$ and let $y = h(w)$. Let $\widetilde \sigma \in
\{<,=,>\}^m$ be  such that $f_i(w) \widetilde \sigma_i 0$ for $1\le
i \le m$. Then, the conditions $y_1 = w_1 - z_1\ge 0,  y_{n-t+1}
\widetilde \sigma_1 0, \dots,
 y_n
\widetilde\sigma_t 0$ hold. Since $w\in C$, for $1\le i \le m$,
 $\widetilde
\sigma_i\in \{<, =\}$ if $\sigma_i= ``\le"$, $\widetilde \sigma_i\in
\{>, =\}$ if $\sigma_i = ``\ge"$ and $\widetilde \sigma_i =
\sigma_i$ otherwise. Hence, every point satisfying $\widetilde
\sigma$ also satisfies $\sigma$.

Let $\gamma: [-\varepsilon /2, y_1] \to (-\varepsilon,
\varepsilon)^n$  be defined as $\gamma (u) = (u, y_2, \dots, y_n)$.
For $u \in [-\varepsilon /2, y_1]$ and $1 \le i \le t$, $f_i( g
\circ \gamma(u)) = y_{n-t+i}\widetilde \sigma_i 0$. Taking into
account that, for $t+1\le i \le m$, $f_i$ has constant sign over $U$
and the image of $g\circ \gamma$ lies in $U$, we also have that $f_i
(g\circ \gamma(u)) \widetilde \sigma_i 0$ for $t+1\le i \le m$.
Therefore, the image of $g\circ \gamma$ is contained in the
realization of $\sigma$ and, since it is a connected curve with a
point $g \circ \gamma(y_1) = w $ in the connected component $C$, we
conclude that it is contained in $C$. Now, the first coordinate of
$g \circ \gamma(-\varepsilon /2)$ is $-\varepsilon /2 + z_1< z_1$,
contradicting the fact that $z_1 = \inf \pi_1(C)$.
\end{proof}

\subsection{Deformation techniques for bihomogeneous systems}
\label{algoritmo_resolucion_sistemas}

In this subsection we present briefly a symbolic deformation
introduced in \cite{GHMMP98}, \cite{GHHM+}, \cite{HKPSW},
\cite{GLS01} and \cite{Schost03}, adapted to the bihomogeneous
setting following \cite{HJSS}.

\subsubsection{The deformation}\label{sistemainic}

Given polynomials $h_1(x),\dots, h_s(x), h_{s+1}(x, \mu), \dots,
h_r(x,\mu)\in\K[x_1,\dots, x_n,\mu_1,\dots, \mu_s]$ we consider the
associated equation system:
\begin{equation}\label{sistema_final}
h_1(x) =0,\dots, h_s(x)  =  0, h_{s+1}(x,\mu)=0,\dots, h_r(x,\mu)
=0,\end{equation} where $ 2 \le s \le n-1$ and $r = s+n-1$ such
that, for $1\le i \le r$,  $\deg_x(h_i) \le d_i \le d$ and, for
$s+1\le i \le r$, $h_i$ is homogeneous of degree $1$ in the
variables $\mu$.

Let $W \subset \A^n \times \P^{s-1}$ be the variety this system
defines. By the multihomogeneous B\'ezout theorem (\cite[Ch.~4,
Sec.~2.1]{Shafarevich}), the degree of $W$ is bounded by
\begin{equation}\label{bezout_number}
D:=\Big(\prod_{i = 1}^s d_i\Big) \Big(\sum_{E \subset \{s+1, \dots,
r\}, \# E = n-s} \prod_{j \in E}d_j\Big) \le \binom{n-1}{s-1}d^n.
\end{equation}

Let $g_1(x), \dots,g_s(x),g_{s+1}(x, \mu),\dots,g_r(x,
\mu)\in\K[x_1,\dots, x_n,\mu_1,\dots, \mu_s]$ be polynomials with
$\deg_x(g_i) = d_i$ for $1\le i \le r$ and homogeneous of degree $1$
in the variables $\mu$ for $s+1\le i \le r$, such that:
\begin{itemize}
\item[(H)] $g_1,\dots, g_r$ define a $0$-dimensional variety in $\A^n
\times \{\mu_s\ne 0\} \subset \A^n \times \P^{s-1}$ with $D$ points
$s_1,\dots, s_D$ satisfying $\pi_x(s_i) \ne \pi_x(s_j)$ for $i\ne
j$, and the Jacobian determinant of the polynomials obtained from
$g_1,\dots, g_r$ by dehomogenizing them with $\mu_s=1$ does not
vanish at any of these points.
\end{itemize}
We will specify polynomial systems meeting these conditions in
Definitions \ref{sistemafactorizable} and \ref{sistemaTcheby} below.

Let $t$ be a new variable. For every $1\le i \le  r$, let
\begin{equation}\label{sistemadeformado}
F_i := (1-t) h_i +t g_i . \end{equation} Consider the variety $\hat
V \subset \A^1 \times \A^{n} \times \P^{s-1}$ defined by $F_1,
\dots, F_r$, and write
\begin{equation}\label{hatV}
\hat V = V^{(0)} \cup V^{(1)} \cup V, \end{equation}
 where $V^{(0)}$ is the union of
the irreducible components of $\hat V$ contained in $\{t = 0\}$,
$V^{(1)}$ is the union of its irreducible components contained in
$\{t = t_0\}$ for some $t_0 \in \C\setminus\{0\}$, and $V$ is the
union of the remaining irreducible components of $\hat V$.

\begin{lemma}\label{finitud}
With our previous assumptions and notation, $\pi_{x,\mu}(V \cap
\{t=0\})$ is a finite subset of $W$ containing all its isolated
points.
\end{lemma}

\begin{proof}{Proof:}
Let $V_1$ be an irreducible component of $V$ and $\overline{V}_1$ be
its Zariski closure in $\A^1\times\P^n\times\P^{s-1}$. The
projection of $\overline{V}_1$ to $\A^1$ is onto and so,
$\overline{V}_1\cap \{t=1\}\ne \emptyset$.
 But our assumption on
$g_1, \dots, g_r$ implies that $\overline{V}_1 \cap \{t=1\} =  V_1
\cap \{t = 1\}$; then, $V_1\cap \{t=1\} \ne \emptyset$ and it is
$0$-dimensional.  It follows that $\dim(V_1) = 1$. Therefore,  $V$
is a $1$-equidimensional variety and thus $\pi_{x,\mu}(V \cap
\{t=0\})$ is a finite set.

In order to prove the second part of the statement, note that $W=
\pi_{x,\mu}(\hat V \cap \{ t=0\}) = \pi_{x,\mu}(V^{(0)}) \cup
\pi_{x,\mu}(V \cap \{ t=0\})$. Now, an isolated point of $W$ cannot
belong to $\pi_{x,\mu}(V^{(0)})$, since the dimension of each of its
irreducible components is at least $1$; hence, it lies in
$\pi_{x,\mu}(V\cap \{t=0\})$.
\end{proof}

The same deformation can be applied to a system $h_1(x), \dots,
h_n(x) \in \K[x_1,\dots, x_n]$ with $g_1(x), \dots, g_n(x) \in
\K[x_1,\dots, x_n]$ such that  $\deg(g_i) = \deg (h_i)$ for every $1
\le i \le n$ and having $\prod_{i=1}^n \deg(g_i)$ common zeros in
$\A^n$.

\subsubsection{A geometric resolution}

\begin{lemma}\label{prop_var_t_como_coef} The variety  defined in
$\A^n_{\overline{\K(t)}} \times \P_{\overline{\K(t)}}^{s-1}$ by
$F_1, \dots, F_r$ as in (\ref{sistemadeformado}) is 0-dimensional
and has  $D$  points $S_1,\dots, S_D$ in $\{\mu_{s}\ne 0\}$ such
that $\pi_x(S_i)\ne \pi_x(S_j)$ for $i\ne j$. Moreover, these points
can be considered as elements in $\K [[t-1]]^r$.
\end{lemma}

\begin{proof}{Proof:} The multihomogeneous  B\'ezout Theorem (see, for instance,
\cite[Chapter 4, Section 2.1]{Shafarevich}) states that the degree
of the variety is bounded by $D$. If $s_i$, $1\le i\le D$, are the
common zeros of $g_1,\dots, g_s, g_{s+1}, \dots, g_r$, the Jacobian
of  $F_1, \dots, F_r$ with respect to $x_1,\dots, x_n, \mu_1,\dots,
\mu_{s-1}$ at $t=1$ and $(x, \mu) = s_i$ is nonzero. The result
follows applying the Newton-Hensel lifting (see for example
\cite[Lema 3]{HKPSW}).
\end{proof}

Consider now new variables $y_1, \dots, y_n$ and define $\ell(x,
\mu, y) = \ell(x, y) = \sum_{j = 1}^n y_j x_j$. For $\alpha_1,
\dots, \alpha_n \in \C$,  let $\ell_{\alpha}(x, \mu) =
\ell_{\alpha}(x) = \sum_{j = 1}^n \alpha_j x_j$. Let
\begin{equation}\label{expresion_resumen}
P(t, U, y) = \prod_{i =1}^D \big(   U -   \ell(S_i, y) \big)  =
\frac{\sum_{h=0}^D p_h(t, y)U^h}{q(t)}=\frac{\hat P(t, U,
y)}{q(t)}\in \mathbb{K}(t)[U, y],
\end{equation}
with $\hat P(t, U, y) \in \K[t, U, y]$ with no factors in
$\K[t]\setminus \K$.

In order to compute $P$ we will approximate its roots. The required
precision is obtained from the following upper bound for the degree
of its coefficients. A similar result in the general sparse setting
appears in \cite[Lemma 2.3]{JMSW}, but to avoid a possibly
cumbersome translation of our setting into sparse systems, we give
an alternative statement and its proof here.

\begin{lemma} \label{gradodeP}Using the notation in (\ref{expresion_resumen}), $\deg_t\hat P(t,U,y)\le nD$.
\end{lemma}

\begin{proof}{Proof: }
Let  $\Phi: V \times \A^n \to \A^{n+2}$ be the morphism defined by
$\Phi(t, x, \mu, y) = (t, \ell(x, y), y)$. It is easy to see that
$\hat P(t, U, y)$ is a square-free polynomial defining
$\overline{{\rm Im} \Phi}$.

For  $\beta = (\beta_0, \beta_1, \dots, \beta_n) \in \C^{n+1}$, let
$\hat P_\beta(t) = \hat P(t, \beta_0, \beta_1, \dots, \beta_n)$. For
a generic $\beta$, $\hat P_\beta(t)$
 is squarefree and satisfies $ \deg_t
\hat P(t, U, y)= \deg_t \hat P_\beta(t)=\# \big( \overline{{\rm Im}
\Phi}  \cap  \{U = \beta_0, y_1 = \beta_1,$ $ \dots, y_n =
\beta_n\}\big)$.

Since $\overline{{\rm Im} \Phi} \setminus {\rm Im} \Phi$ is
contained in an $n$-dimensional variety, for a generic $\beta$ we
have $\overline{{\rm Im} \Phi} \cap \{U = \beta_0, y_1 = \beta_1,
\dots, y_n = \beta_n\} = {\rm Im} \Phi \cap \{U = \beta_0, y_1 =
\beta_1, \dots, y_n = \beta_n\}$. Moreover, for a  fixed $t_0 \in
\C$, for a generic  $\beta$ we also have $\overline{{\rm Im} \Phi}
\cap \{t = t_0\} \cap \{U = \beta_0, y_1 = \beta_1, \dots, y_n =
\beta_n\} = \emptyset$. Then, for a finite set $T\subset \C$, for a
generic $\beta$, $\overline{{\rm Im} \Phi}  \cap  \{U = \beta_0, y_1
= \beta_1, \dots, y_n = \beta_n\}
 $
$ = {\rm Im} \Phi  \cap  \big(\cap_{t_0 \in T}\big\{t \ne
t_0\big\}\big)  \cap  \{U = \beta_0, y_1 = \beta_1, \dots, y_n =
\beta_n\}.$

Let $T = \{t_0 \in \C \ | \ \hat V$ has an irreducible component
included in $\{t = t_0\}\}$ and $\beta$ generic satisfying all the
previous conditions.

For each $(\bar t, \beta)  \in {\rm Im} \Phi \cap (\cap_{t_0 \in
T}\{t \ne t_0\}) \cap \{U = \beta_0, y_1 = \beta_1, \dots, y_n =
\beta_n\}$, there exists a point $(\bar t, \bar x, \bar \mu) \in V$
such that $\ell_{(\beta_1, \dots, \beta_n)}(x) = \beta_0$. Then,
$(\bar t, \bar x, \bar \mu)$ is a solution to the system
\begin{equation}\label{sistema con l}
F_1(t, x) = 0, \dots, F_s(t, x) = 0, F_{s+1}(t, x, \mu)
 = 0, \dots,F_r(t, x, \mu)  =  0,\ell_{(\beta_1, \dots, \beta_n)}(x) - \beta_0  =
 0
\end{equation}
with $\bar t \notin T$. Since $\hat P_\beta$ is a nonzero
polynomial, $V$ has no irreducible components included in
$\{\ell_{(\beta_1, \dots, \beta_n)}(x) = \beta_0\}$ and, therefore,
$V \cap \{\ell_{(\beta_1, \dots, \beta_n)}(x) = \beta_0\}$ is a
$0$-dimensional variety. This implies that the point $(\bar t, \bar
x, \bar \mu)$ is an isolated solution of the system (\ref{sistema
con l}). Then, in order to get an upper bound for $\deg_t\hat
P(t,U,y)$ it suffices to bound the number of isolated roots of this
system.

Using the multihomogeneous B\'ezout theorem in the three groups of
variables $x$, $\mu$ and $t$, it follows that this system has at
most $ \sum_{(E_t, E_x, E_\mu)} \ \prod_{k \in E_x}  d_k
$
isolated roots, where $d_{r+1} := 1$ and $(E_t, E_x, E_\mu)$ runs
over all the partitions of $\{1, \dots, r+1\}$ into sets $E_t, E_x$
and $E_\mu$ of cardinality $1, n$ and $s-1$ respectively, with $E_t
\subset \{1, \dots, r\}$ and $E_\mu \subset \{s+1, \dots, r\}$. Each
of these partitions $(E_t, E_x,E_\mu)$ can be obtained uniquely from
a subset $E\subset\{s+1, \dots, r\}$ of cardinality  $n-s$ and an
element $e \in \{1, \dots, s\} \cup E$ by taking $E_t = \{e\}, E_x =
(\{1, \dots, s, r+1\} \cup E) \setminus \{e\}$ and $E_\mu = \{s+1,
\dots, r\} \setminus E$. Therefore,
$$
\sum_{(E_t, E_x, E_\mu)} \ \prod_{k \in E_x}  d_k \ = \ \sum_{(E,
e)} \ \prod_{k \in (\{1, \dots, s, r+1\} \cup E) \setminus \{e\}}
 d_k \ \le \ \sum_{(E, e)} \ \prod_{k \in \{1, \dots, s\} \cup
E}  d_k \ \le \ nD.$$
\end{proof}

Note that, if $\pi_{x}(V \cap \{t=0\})=\{z_1, \dots, z_\nu\} \subset
\A^n$, for every $1\le l \le \nu$, $U - \ell(z_l, y)$ divides $\hat
P(0, U, y)$. Then, if $\hat P(0, U, y) = \prod_{j=1}^a q_j(U,
y)^{\delta_j}$ is the factorization of $\hat P(0, U, y)$ in $\C[U,
y]$, we may suppose that $q_l = U - \ell(z_l, y)$ for every $1\le
l\le \nu\le a$. Let  $Q(U, y) = \prod_{j=1}^{a} q_j^{\delta_j-1}$.

Let $\alpha = (\alpha_1, \dots, \alpha_n) \in \C^n$ such that, for
$1 \le j \le a$, $q_j(U, \alpha)$ has the same degree in $U$ as
$q_j(U, y)$ and is square-free and, for $1 \le j_1 < j_2 \le a$,
$q_{j_1}(U, \alpha)$ and $q_{j_2}(U, \alpha)$ are relatively prime
polinomials in $\C[U]$. Note that these conditions are met for a
generic vector $\alpha$. Then, $\frac{\hat P(0, U, \alpha)}{Q(U,
\alpha)}$ is square-free and vanishes at $\ell(z_l, \alpha)$ for
$1\le l \le \nu$. In addition, for $1\le l\le \nu$, the $k$th
coordinate of $z_l$ $(1\le k \le n)$ is the quotient of
$$-\frac{\frac{\partial \hat P}{\partial y_k}(0, \ell(z_l, \alpha),
\alpha)}{Q(\ell(z_l, \alpha), \alpha)}  =  \delta_l (z_l)_k \prod_{j
\ne l} q_{j}(\ell(z_l, \alpha), \alpha).
$$ by
$$
 \frac{\frac{\partial \hat P}{\partial U}(0, \ell(z_l, \alpha), \alpha)}{Q(\ell(z_l, \alpha), \alpha)}
 = \delta_l \prod_{j \ne l} q_{j}(\ell(z_l, \alpha), \alpha)
\ne 0. $$ We deduce the following:

\begin{proposition} \label{resoluc_geom_buscada} Let $\hat P(t,U,y)$ be as in
(\ref{expresion_resumen}). Then, for a generic $\alpha\in \C^n$,
$$\Big\{\frac{\hat P(0, U, \alpha)}{Q(U, \alpha)}; \frac{\frac{\partial
\hat P}{\partial U}(0, U, \alpha)}{Q(U, \alpha)}; -
\frac{\frac{\partial \hat P}{\partial y_1}(0, U, \alpha)}{Q(U,
\alpha)}, \dots, - \frac{\frac{\partial \hat P}{\partial y_n}(0, U,
\alpha)}{Q(U, \alpha)}\Big\}$$ is a geometric resolution of a finite
set containing $\pi_x(V \cap \{ t=0\})$. \hfill $\square$
\end{proposition}

\section{Regular intersections}
\label{seccion_hipotesis}\label{situacion_con_hipotesis}\label{desigualdades_estrictas}

Let  $f_1,\dots, f_m\in \R[x_1,\dots, x_n]$ be polynomials meeting
the following condition:

\begin{hypothesis}\label{la_hipotesis} For every $x \in \C^n$, for every $ \{i_1, \dots, i_s\}
\subset \{1, \dots, m\}$, if $f_{i_1}(x) = \dots = f_{i_s}(x) = 0$,
then $\{\nabla f_{i_1}(x), \dots, \nabla f_{i_s}(x)\}$ is linearly
independent.
\end{hypothesis}

As in Section \ref{definicion_sistemas}, for $S = \{i_1,\dots, i_s\}
\subset \{1, \dots, m\}$, consider the solution set $W_S$ of the
system (\ref{J_S}), (\ref{cond_1}) or (\ref{cond_s_grande})
depending on whether $2\le s\le n-1$, $s=1$ or $s\ge n$
respectively. Note that, under Hypothesis \ref{la_hipotesis}, $W_S$
is the empty set whenever $s>n$. Moreover:

\begin{lemma}\label{conjunto_candidato_finito}
Under Hypothesis \ref{la_hipotesis}, after a generic linear change
of variables, for every $S\subset \{1,\dots, m\}$, the set $W_S$ is
finite.
\end{lemma}

\begin{proof} {Proof:} If $s = n$, Hypothesis \ref{la_hipotesis}
implies that $W_S$ is a finite set.

Now let $s\le n-1$. Note that $\pi_x(W_S)$ is the set of critical
points of the map $(x_1,\dots, x_n)\mapsto x_1$ over the set $\{x\in
\C^n \mid f_{i}(x)= 0  \hbox{ for } i \in S\}$. By the arguments in
\cite[Section 2.1]{Voisin} based on Sard's theorem and a holomorphic
Morse lemma, it follows that a generic linear form has a finite
number of critical points on this  complex variety. Therefore,
taking any of these generic linear forms as the first coordinate,
the set $\pi_x(W_S)$ turns to be finite. Moreover, for every $z\in
\pi_x(W_S)$, since $\{\nabla f_{i}(z),\ i\in S\}$ is linearly
independent and $\{\overline \nabla f_{i}(z),\ i\in S\}$ is linearly
dependent, it follows that there is a unique $ \mu \in \P^{s-1}$
such that $(z, \mu)  \in W_S$.
\end{proof}

Given $f_1,\dots, f_m\in \K[x_1,\dots, x_n]$, we will deal with the
deformation (\ref{sistemadeformado}) and the corresponding varieties
(\ref{hatV}) defined from the systems (\ref{J_S}), (\ref{cond_1})
and (\ref{cond_s_grande}) for $S\subset \{1,\dots, m\}$ with $1\le
\#S\le n$, and an adequate initial system. We will add a subscript
$S$ in the notation $\hat V$, $V^{(0)}$, $V^{(1)}$ and $V$ to
indicate that the varieties are defined from the polynomial system
associated with $S$.

{}From Proposition \ref{karush_kuhn_tucker}, Lemma \ref{finitud} and
Lemma \ref{conjunto_candidato_finito} we deduce:

\begin{proposition}\label{alcanza_con_de_a_n}
Let $f_1,\dots, f_m\in \R[x_1,\dots, x_n]$ be polynomials satisfying
Hypothesis \ref{la_hipotesis}. Let $\sigma \in \{<, =,
>\}^m$ and $E_\sigma = \{i \mid \sigma_i = ``=" \}$. Then, after a generic linear change of
variables, for each connected component $C$ of $\{x \in \R^n \mid
f_{1}(x) \sigma_{1} 0, \dots, f_m(x) \sigma_{m} 0\}$,
$$Z(C) \subset \bigcup_{{E_\sigma \subset S \subset \{1, \dots,
m\}}\atop{1 \le \# S \le n}}\pi_x\big(V_S \cap \{t=0\}\big).$$
\hfill $\square$
\end{proposition}

\subsection{Initial systems for deformations}

In what follows we introduce  the initial systems for our first
algorithmic deformation procedure.

\begin{defn}\label{sistemafactorizable} For a given $s$ with $1< s <n$ and $r:=
s+n-1$, a \emph{type 1 initial system} is a polynomial system of the
form:
$$
\left\{\begin{array}{ll}
g_i(x) =  \displaystyle\prod_{1 \le j \le d_i} (x_i - j) & \hbox{ for } 1\le i \le  s, \\[0.4cm]
g_i(x, \mu) = \Big(\displaystyle\prod_{1 \le j\le
d_i}\phi_{ij}(x)\Big)\psi_{i}(\mu)  & \hbox{ for } s+1\le i \le r,
\end{array}\right.
$$
where, for $s+1\le i \le r$, $\phi_{ij}(x) = \Big(\sum_{s+1 \le k
\le n} \frac{1}{(i-s-1)d + j-1 + k-s }x_k \Big) + \frac{1}{(i-s-1)d
+ j-1 + n+1 -s}$ $(1 \le j \le d_i)$,   and $ \psi_{i}(\mu) =
\sum_{1 \le k \le s}\frac{1}{i - s - 1 + k}\mu_k$.

 For $s=1$ and $s=n$, a type 1 initial system consists of $n$
polynomials of the form $g_i(x) =  \displaystyle\prod_{1 \le j \le
d_i} (x_i - j)$ $(1\le i \le n)$.
\end{defn}

\begin{lemma}\label{sistema_factorizable_apropiado}
Property (H) (see Section \ref{sistemainic}) holds for any type 1
initial system.
\end{lemma}

\begin{proof}{Proof}
Let $\bar g_{s+1}, \dots, \bar g_r$ be the polynomials obtained by
dehomogeneizing $g_{s+1},\dots,g_r$ in the variables $\mu$ with
$\mu_s= 1$ and let $(\bar{x}_1, \dots, \bar{x}_n, \bar{\mu}_1,
\dots, \bar{\mu}_{s-1}) \in \A^{n+s-1}$ be a solution of $g_1 =
\dots = g_s = \bar{g}_{s+1} = \dots = \bar{g}_r =0$. As,  for $1 \le
i \le s$, $\bar{x}_i$ is one of the  $d_i$ zeros of $g_i$,  to count
the solutions of the original system, the $\prod_{1 \le i \le s} d_i
$ solutions $(\bar{x}_1, \dots, \bar{x}_s)$ should be combined with
the solutions $(\bar{x}_{s+1}, \dots, \bar{x}_n, \bar{\mu}_1, \dots,
\bar{\mu}_{s-1})$ of the system $\bar{g}_{s+1} = \dots = \bar{g}_r =
0$. Then,  it suffices to show that this new system has $\sum_{E
\subset \{s+1, \dots, r\}, \# E = n-s} \prod_{k \in E}d_k$ different
zeros in $\A^{n-1}$.

For $s+1 \le i \le r$, $\bar{g}_i$ is a product of $d_i$ linear
forms $\phi_{ij}(x_{s+1}, \dots, x_n)$ and the linear form
$\bar\psi_i(\mu_1, \dots, \mu_{s-1}) := \psi_i(\mu_1, \dots,
\mu_{s-1}, 1)$. Therefore, each $(\bar{x}_{s+1}, \dots, \bar{x}_n,
\bar{\mu}_1, \dots, \bar{\mu}_{s-1})$ must be a zero of at least one
of these $d_i + 1$ linear forms for each $i$. Suppose it is a zero
of $a$ linear forms  $\phi_{ij}(x)$ and  $b \ge r-s-a$ linear forms
$\bar \psi_i(\mu)$. Then, $(\bar x_{s+1}, \dots, \bar x_n, 1)$ is a
nonzero vector in the kernel of a Cauchy matrix in $\Q^{a \times
(n-s+1)}$, and therefore, $n-s \ge a$. On the other hand,
$(\bar\mu_1, \dots, \bar\mu_{s-1}, 1)$ is a nonzero vector in the
kernel of another Cauchy matrix in $\Q^{b \times s}$, and therefore
$s-1 \ge b$. As $a + b \ge r-s= n-1 = (n-s) + (s-1) \ge a + b$, then
$a = n-s$ and $b = s-1$. Therefore, for $s+1 \le i \le r$, $(\bar
x_{s+1}, \dots, \bar x_n, \bar\mu_1, \dots, \bar\mu_{s-1})$ is a
zero of only one of the  $d_i+1$ linear forms defining $\bar{g}_i$.

The number $\sum_{E \subset \{s+1, \dots, r\}, \# E = n-s} \prod_{k
\in E}d_k$ arises from choosing a subset of  $n-s$ equations of the
form  $\bar{g}_i = 0$ in which a linear form corresponding to the
variables $x$ vanishes and, then, in each equation, which factor
vanishes. In the other $s-1$ equations, the linear form
corresponding to the $\mu$ variables must vanish. Any system of this
type, determines only one point  in $\A^{n-1}$ because the Cauchy
matrices involved are invertible.

To see that all the solutions obtained in this way are different, it
suffices to show that two such solutions satisfy $(\bar x_{1}^{(1)},
\dots, \bar x_n^{(1)}) \ne (\bar x_{1}^{(2)}, \dots, \bar
x_n^{(2)})$. Otherwise, different choices of the $n-s$ linear forms
$\phi_{i_lj_l}$ would give the same solutions, and this is
impossible for any choice of $n-s+1$ of these linear forms defines
an invertible Cauchy matrix.

Finally,   note that the Jacobian matrix of the system $g_1 ,\dots,
g_s,\bar{g}_{s+1},\dots,\bar{g}_r$ is a diagonal block matrix. Its
upper block is an $s \times s$ diagonal block with nonzero elements
in its diagonal. Without loss of generality, we may suppose that
$(\bar{x}_{s+1}, \dots, \bar x_n)$ is a common zero of the linear
forms $\phi_{(s+1)1}, \dots, \phi_{n 1}$ and that $(\bar{\mu}_1,
\dots, \bar{\mu}_{s-1})$ is a common zero of the linear forms
$\bar\psi_{n+1}, \dots, \bar\psi_{r}$. Then, the lower block is the
product of a diagonal matrix with the elements
$\psi_{s+1}(\bar{\mu}) \displaystyle\prod_{j =
2}^{d_{s+1}}\phi_{(s+1)j}(\bar{x}), \dots,\psi_{n}(\bar{\mu})
\displaystyle\prod_{j = 2}^{d_n}\phi_{n j}(\bar{x}),
\displaystyle\prod_{j = 1}^{d_{n+1}}\phi_{(n+1)j}(\bar{x}), \dots,
\displaystyle\prod_{j = 1}^{d_{r}}\phi_{rj}(\bar{x})$ in its
diagonal by a block diagonal matrix whose blocks are the Cauchy
matrices formed by the coefficients of the linear forms
$\phi_{(s+1)1}, \dots, \phi_{n 1}$ and $\bar\psi_{n+1}, \dots,
\bar\psi_{r}$ respectively, which are both invertible.
\end{proof}

\subsection{Symbolic deformation algorithms}

We describe in this section a probabilistic algorithm which computes
a geometric resolution as in Proposition \ref{resoluc_geom_buscada}.

\begin{proposition} \label{algoritmopararesgeom} There is a probabilistic algorithm that, taking as input polynomials
$h_1, \dots, h_{r}$ in $ \K [x_1, \dots, x_n, \mu_1, \dots, \mu_s]$
as in (\ref{sistema_final}) encoded by an slp of length $L$, obtains
a geometric resolution of a finite set containing $\pi_x(V \cap \{
t=0\})$ for a deformation defined from a type 1 initial system
within complexity $O\big( n^2D^2\log(D)\log\log(D)\big( 
L+ \log^2(D)\log\log(D) \big)\big).$
\end{proposition}

\begin{proof}{Proof: } The procedure of this algorithm is standard.
The main difference with previous known algorithms solving this task
(see, for example, \cite{GLS01} or \cite{HKPSW}) is that the Newton
lifting will be done pointwise.

\medskip

\noindent {\sc{First step:}} Form a type 1 initial system of
polynomials of the same degree structure as $h_1,\dots, h_r$ and
compute the solutions $s_1,\dots, s_D$ of this system. The
computation of each solution amounts to solving two square linear
systems of size $n-s$ and $s-1$ respectively with Cauchy matrices,
which can be done within a complexity $O(n \log^2(n))$ by means of
\cite[Ch.~2, Algorithm 4.2]{BP94}).

\medskip

\noindent {\sc Second step:} Construct an slp encoding $F_1,\dots
,F_r$ (see (\ref{sistemadeformado})). Since $g_1,\dots, g_r$ can be
encoded by an slp of length $O(dn^2)$, the length of this slp can be
taken to be $L_1 = L + O(dn^2)$. Set $F$ for the list of polynomials
$F_1,\dots F_r$ dehomogenized with $\mu_s =1$. The algorithm
computes, for $i = 1, \dots, D$, elements $\tilde S_i \in \K[t]^{r}$
such that for $1\le k\le r,$ $(\tilde S_i -S_i)_k \in (t-1)^{2nD +
1}\K[[t-1]]$: Let $\tilde S_i^{(0)} = s_i$ be a solution of the
initial system $g_1, \dots, g_r$. By means of the Newton-Hensel
operator we define recursively $ \tilde S_i^{(m+1)} = \tilde
S_i^{(m)} - DF^{-1}(\tilde S_i^{(m)})F(\tilde S_i^{(m)}) \mod
(t-1)^{2^{m+1}}\K[[t-1]].$ For $1\le k\le r$ and $m \in \N_0$,
$(\tilde S_i^{(m+1)})_k \equiv (\tilde S_i^{(m)})_k \hbox{ mod }
(t-1)^{2^m}\K[[t-1]]$, and $(\tilde S_i^{(m)})_k$ is a polynomial in
$t-1$ of degree less than $2^{m}$. Operations between polynomials of
such degree can be done using $O(2^mm\log(m))$ operations in $\K$.

We are now going to estimate the complexity of computing $\tilde
S_i:= \tilde S_i^{(\delta)}$ from $s_i$ for $\delta = \lceil\log(2nD
+1)\rceil$, $1\le i \le D$.

For every $i$, in a first step, we have $\tilde S_i^{(0)} = s_i$ and
matrices $A^{(0)} = Dg(s_i)$ and  $B^{(0)} = Dg^{-1}(s_i)$ in $\K^{r
\times r}$. The computation of these matrices  does not change the
overall complexity. Assume that $A^{(m)}B^{(m)}\equiv I \mod
(t-1)^{2^{m}} \K [[t-1]]$. In the $(m+1)$th step, we evaluate the
matrix $A^{(m+1)} \equiv DF(\tilde S_i^{(m)})  \mod (t-1)^{2^{m+1}}
\K [[t-1]]$. The cost of this step is $O(nL_12^mm\log(m))$. When
$m=0$, we take  $B^{(1)} = 2B^{(0)} - B^{(0)}A^{(1)}B^{(0)}$. For
$m\ge 1$, to define an approximation $B^{(m+1)}$ to the inverse of
$A^{(m+1)}$, noticing that $A^{(m+1)} \equiv DF(\tilde S_i^{(m)})
\equiv DF(\tilde S_i^{(m-1)}) \equiv A^{(m)} \mod
(t-1)^{2^{m-1}}\K[[t-1]]$, we compute $B' = B^{(m)} - B^{(m)}
(A^{(m+1)} - A^{(m)}) B^{(m)}.$ We have that $A^{(m+1)}B' \equiv I
\mod (t-1)^{2^{m}}\K[[t-1]]$. Then, $B^{(m+1)} = 2 B' - B'A^{(m+1)}
B'$ is obtained, which satisfies $A^{(m+1)} B^{(m+1)} \equiv I  \mod
(t-1)^{2^{m+1}}\K[[t-1]]$. The complexity of this step is
$O(n^32^mm\log(m))$. Finally we evaluate $\tilde S_i^{(m+1)} =
\tilde S_i^{(m)} - B^{(m+1)}F(\tilde S_i^{(m)})$ within $O(
L_12^mm\log(m))$ operations.

The total number of operations over $\K$ to do this step for every
$1 \le i \le D$  is $O(n(nL_1 + n^{3})D^2\log(D)\log\log(D))$.

\medskip

\noindent{\sc Third step:} This step consists in the computation of
$\hat P(0,U, \alpha) = \sum_{h=0}^D p_h(0,\alpha) U^h$ and
$\frac{\partial \hat P}{\partial y_k}(0,U, \alpha) = \sum_{h=0}^D
\frac{\partial p_h}{\partial y_k}(0,\alpha) U^h$ for a generic
$\alpha = (\alpha_1, \dots, \alpha_n) \in \Q^n$.

By expanding $ P(t, U, y) = \sum_{h=0}^D  \frac{p_h(t, y)}{q(t)}U^h
\in \K[[t-1]][U, y]$ into powers of  $U$, $(y_1 - \alpha_1), \dots,
(y_n - \alpha_n)$, we have that the coefficients corresponding to
$U^h$ and $U^h(y_k - \alpha_k)$ $(1\le k \le n,\, 0\le h \le D)$ are
$ {p_h(t, \alpha)}/{q(t)}$ and ${\frac{\partial p_h}{\partial
y_k}(t, \alpha)}/{q(t)}$ respectively. As the degrees of the
polynomials involved in these fractions are bounded by $nD$ (see
Lemma \ref{gradodeP}), they are uniquely determined by their power
series expansions modulo $(t-1)^{2nD + 1}\K[[t-1]]$ (see
\cite[Corollary 5.21]{vzG}).

The algorithm proceeds as follows: first, it computes the
coefficients of $U^h$ and $U^h(y_k - \alpha_k)$ $(1\le k \le n,\,
0\le h \le D)$ in $ \tilde P(t, U, y) = \prod_{i=1}^D (U-\ell(\tilde
S_i, y)) \in \K[t][U, y]$ following \cite[Algorithm 10.3]{vzG} in
$O(n^2D^2\log^3(D)\log\log^2(D))$ operations over $\K$. From these
coefficients, $p_h(t, \alpha)$ and $\frac{\partial p_h}{\partial
y_k}(t, \alpha)$ $(1\le k \le n,\, 0\le h \le D)$, and $q(t)$ are
obtained within complexity $O(n^2D^2\log^2(D) \log \log(D))$ over
$\K$ by using \cite[Corollary 5.24 and Algorithm 11.4]{vzG} and
converting all rational fractions to a common denominator. Finally,
the algorithm substitutes $t=0$ in these polynomials  to obtain
$\hat P(0, U, \alpha)$ and $\frac{\partial \hat P}{\partial
y_k}(0,U, \alpha)$ for $1\le k\le n$.

\medskip

\noindent {\sc Fourth step:} The algorithm computes $Q(U, \alpha) =
\textrm{gcd}(\hat P(0, U, \alpha),\frac{\partial \hat P}{\partial
U}(0,U, \alpha))$ within complexity $O(D \log^2(D) \log \log (D))$
and makes the required exact divisions by $Q(U, \alpha)$ leading to
the geometric resolution. This last step does not change the overall
order of complexity, which is $O\big( n^2D^2\log(D)\log\log(D)\big(
L + n^2d + \log^2(D)\log\log(D) \big)\big).$
\end{proof}

The algorithm underlying the proof of Proposition
\ref{algoritmopararesgeom} can be adapted straightforwardly to
handle the cases $s=1$ and $s = n$ within the same complexity
bounds.

\subsection{Main algorithm}

The main algorithm of this section is the following:

\begin{algor}\label{algoritmo_general} \hspace{1 cm} \linebreak \vspace{-0.5cm}

\medskip
\noindent \emph{Input}: Polynomials $f_1, \dots, f_m\in
\K[x_1,\dots, x_n]$ satisfying Hypothesis \ref{la_hipotesis} encoded
by an slp of length $L$,
 and positive integers $d_1,\dots, d_m$ such that $\deg f_i\le d_i$ for
$1\le i\le m$.

\smallskip
\noindent \emph{Output}: A finite set ${\cal M}\subset \A^n$
intersecting the closure of each connected component of the
realization of every feasible sign condition over $f_1,\dots, f_m$
encoded by a list $\mathcal{G}$ of geometric resolutions of
$0$-dimensional varieties.

\medskip
\noindent \emph{Procedure}: \begin{enumerate}

\item Make a random linear change of variables with coefficients in $\Q$.

\item Take a point $p = (p_1, \dots, p_n) \in \Q^n$ at random.

\item Starting with $\mathcal{G}:=\emptyset$, for $k=1, \dots, n-1$ and for every
$S \subset \{1, \dots, m\}$ with $1 \le \#S \le n - k + 1$:
\begin{enumerate} \item Obtain an slp encoding the polynomials
which define the variety $W_{k,S}$ associated with the polynomials
$f_1(p_1, \dots, p_{k-1}, x_k, \dots, x_n), \dots,$ $f_m(p_1, \dots,
p_{k-1}, x_k, \dots, x_n)$ and the projection to the $k$th
coordinate $x_k$.
\item Compute a geometric resolution $\{q^{(k, S)}(U), v_k^{(k, S)}(U), \dots,
v_n^{(k, S)}(U)\}\subset \K[U]$ of the variety $\pi_x(V_{k,S} \cap
\{t = 0\}) \subset \A^{n-k+1}$ by means of a deformation from a type
$1$ initial system and add the geometric resolution
$$\Big\{q^{(k, S)}(U), p_1, \dots, p_{k-1}, v_k^{(k, S)}(U), \dots,
v_n^{(k, S)}(U)\Big\}$$ to the list $\mathcal{G}$.
\end{enumerate}
\item Add to the list $\mathcal{G}$ the geometric resolutions
$\{f_{i}(p_1, \dots, p_{n-1}, U), p_1, \dots, p_{n-1}, U\}$, for $1
\le i \le m$, and $\{U, p_1, \dots, p_n\}.$
\end{enumerate}
\end{algor}

\begin{theorem} \label{algoritmopuntosencomponentes}
Algorithm \ref{algoritmo_general} is a probabilistic procedure that,
from a family of polynomials $f_1, \dots, f_m \in \K[x_1, \dots,
x_n]$ satisfying Hypothesis \ref{la_hipotesis}, obtains a finite set
$\mathcal{M}$ intersecting the closure of each connected component
of the realization of every sign condition over $f_1,\dots, f_m$. If
the input polynomials have degrees bounded by $d\ge 2$ and are
encoded by an slp of length $L$, the algorithm performs
$O\big(\big(\sum_{s = 1}^{\min\{m, n\}}
\binom{m}{s}\binom{n-1}{s-1}^2\big)d^{2n} n^4\log(d)$ $(\log(n) +
\log\log(d))\big( L + n^2+  n \log^2(d) (\log(n)+\log \log (d))
\big)\big)$ operations in $\K$.
\end{theorem}

\begin{proof}{Proof: }
Assuming that the random linear change of variables made in the
first step of the algorithm  is generic in the sense of Proposition
\ref{cambio_var_gen_finit_extrem}, by Proposition
\ref{idea_algoritmo}, it suffices to show that
$$\label{correctitud_regular}\{p\} \cup \Big( \bigcup_{k=1}^{n}
\bigcup_{C \in {\cal C}(k,p)} Z(C,k) \Big) \subset \mathcal{M}
$$
where ${\cal C}(k, p)$ denotes the set of all the connected
components of the $\R^n$-subsets  $\Gamma\cap \{x_1 = p_1, \dots,
x_{k-1}=p_{k-1}\}$ with $\Gamma$ a connected component of a feasible
sign condition over $f_1, \dots, f_m$.

Note that for a generic point $p= (p_1, \dots, p_n)\in \K^n$, for
every $2\le k \le n$, Hypothesis \ref{la_hipotesis} also holds  for
the polynomials $f_i(p_1, \dots, p_{k-1}, x_k, \dots, x_n)$, $1\le i
\le n$. Thus,  by taking $p$ at random, we may assume that the
hypothesis is met at each step of the recursion.

Then, for every $1\le k \le n-1$, by Proposition
\ref{alcanza_con_de_a_n} we have that $$\bigcup_{C \in {\cal
C}(k,p)} Z(C,k) \subset \bigcup_{{S \subset \{1, \dots, m\}}\atop{1
\le \# S \le n-k+1}}\{p_1,\dots, p_{k-1}\} \times \pi_x\big(V_{k,S}
\cap \{t=0\}\big),$$ and Step $3$ of the algorithm computes
geometric resolutions for the sets in the right-hand side union;
therefore, $\bigcup_{C \in {\cal C}(k,p)} Z(C,k)\subset
\mathcal{M}$. Finally, note that $\bigcup_{C \in {\cal C}(n,p)} Z(C,
n) \subset \bigcup_{i=1}^m\{ f_i(p_1,\dots, p_{n-1}, x_n)= 0\}$,
which along with the point $p$, is added to the set $\mathcal{M}$ in
Step $4$ of the algorithm. This proves the correctness of Algorithm
\ref{algoritmo_general}.

For every $1\le k \le n-1$ and each $S\subset \{1,\dots, m\}$ of
cardinality at most $n-k+1$, the slp encoding the polynomials which
define the variety $W_{k,S}$ computed at Step $2$ of the algorithm
can be taken of length $O(nL+n^3)$. Moreover, the number of points
in $V_{k,S} \cap \{t = 0\}$ is bounded by $\binom{n-k}{s-1}
d^{n-k+1}$. Therefore, the result follows using the complexity
estimate in Proposition \ref{algoritmopararesgeom}.
\end{proof}

Now we will show how to get the entire list of feasible sign
conditions over the polynomials $f_1, \dots, f_m$ satisfying
Hypothesis \ref{la_hipotesis}. The algorithm relies on the
following:

\begin{proposition}\label{lista_completa_cond_signo}
Let $f_1,\dots, f_m\in \K[x_1,\dots, x_n]$ be polynomials satisfying
Hypothesis \ref{la_hipotesis} and let $\mathcal{M}$ be  a finite set
such that $\mathcal{M} \cap \overline C \ne \emptyset$ for every
connected component $C$ of the realization of each feasible sign
condition over $f_1,\dots, f_m$. Then, the set of all feasible sign
conditions over $f_1,\dots, f_m$  is $ \bigcup_{\sigma \in
\mathcal{L}(\mathcal{M})} P_{\sigma}$ where
$\mathcal{L}(\mathcal{M})$ is the set of all sign conditions
satisfied by the elements of $\mathcal{M}$ and  $P_{\sigma}$ denotes
the subset of $\{<, = ,
>\}^m$ consisting of the elements that can be obtained from $\sigma$
by replacing some of its ``$=$" coordinates with ``$<$" or ``$>$".
\end{proposition}

\begin{proof}{Proof: } Let $\hat\sigma $ be a feasible sign condition and $C$  a connected
component of $\{ x\in \R^n \mid f_1(x) \hat\sigma_1 0, \dots, f_m(x)
\hat\sigma_m 0\}$. Consider a point $z \in {\cal M} \cap \overline
C$ and let $\sigma \in \mathcal{L}(\mathcal{M})$ be the sign
condition over $f_1,\dots f_m$ at $z$. By continuity, it follows
that $\hat\sigma \in P_{\sigma}$.

Now, let $\sigma \in \mathcal{L}(\mathcal{M})$ and $z\in
\mathcal{M}$ such that $f_i(z) \sigma_i 0$ for $1\le i \le m$.
Without loss of generality, assume $\sigma = (=, \dots, =,
>, \dots,
>)$ with $t$ ``$=$" and $m-t$ ``$>$". If $t = 0$, $P_{\sigma} = \{\sigma\}$.
Suppose now $t>0$, and let $\hat\sigma \in P_{\sigma}$. We may
assume $\hat\sigma = (=, \dots, =, >, \dots, >)$ with $l$ ``$=$",
where $0 \le l \le t$. Since the vectors $\nabla f_1(z), \dots,
\nabla f_t(z)$ are linearly independent, there exists $v \in \R^n$
such that $\langle \nabla f_i(z), v \rangle = 0$ for $1\le i\le  l$
and $\langle \nabla f_i(z), v \rangle > 0$ for $l+1\le i \le t$.
Consider a $\mathcal{C}^\infty$ curve $\gamma:[-1, 1] \to \{f_1 =
\dots = f_l = 0\}$ such that $\gamma(0) = z$ and $\gamma'(0) = v$.
For $l+1\le i \le t$,  $f_i \circ \gamma(0) = 0$ and $(f_i \circ
\gamma)'(0) = \langle \nabla f_i(z), v \rangle
> 0$; therefore, for a sufficiently small $u>0$,
$f_i \circ \gamma(u) > 0$ holds. In addition, for  $1\le i\le l$,
$f_i \circ \gamma(u) = 0$ for every $u\in [-1, 1]$. Finally, for
$t+1\le i\le m$, as $f_i\circ \gamma (0) >0$, we have $f_i \circ
\gamma(u) > 0$ for a sufficiently small $u$. We conclude that
$\hat\sigma$ is feasible.
\end{proof}

 Given a geometric resolution $\{q(U),
v_1(U), \dots, v_n(U)\}\subset \K[U]$ consisting of polynomials of
degree bounded by  $\delta$, if $f_1, \dots, f_m$  are encoded by an
slp of length $L$, it is possible to obtain the signs they have at
the points represented by the geometric resolution within complexity
$O\big(L\delta\log(\delta)\log\log(\delta) + m\delta^{\omega}\big)$
(here $\omega\le 2.376$ is a positive real number such that for any
field $k$ it is possible to invert matrices in $k^{r \times r}$ with
$O(r^\omega)$ operations, see \cite{CW}).  First, for $1\le i \le
m$, compute $ f_i(v_1(U),\dots, v_n(U)) \mod q(U) $ within
complexity
 $O(L\delta\log(\delta)\log\log(\delta))$ (\cite[Chapter 8]{vzG})
and then, evaluate the signs of these polynomials at the roots of
$q$ by using the procedure described in  \cite[Section 3]{Canny93}
within complexity  $O(m\delta^{\omega})$.

By applying the previous procedure to the output of Algorithm
\ref{algoritmo_general} and using Proposition
\ref{lista_completa_cond_signo}, we compute the list of all feasible
sign conditions over $f_1, \dots, f_m$.

\begin{theorem}\label{allsignconditions} There is a probabilistic algorithm that, given polynomials
$f_1, \dots, f_m \in \K[x_1, \dots, x_n]$ of degrees bounded by
$d\ge 2$ satisfying Hypothesis \ref{la_hipotesis} and encoded by an
slp of length $L$, computes the list of all feasible sign conditions
over these polynomials within complexity $O \big( \sum_{s =1}^{{\rm
min}\{m, n\}} \binom{m}{s}\big( ( L+n^2 d) {\binom{n-1}{s-1}}^2
d^{2n} n^4 \log(d) (\log(n) + \log\log(d))+ md^{{\omega}
n}\binom{n-1}{s-1}^{\omega}  \big) \big).$ \hfill $\square$
\end{theorem}

Our algorithms and complexity results can be refined if we are
interested in a particular sign condition $\sigma$ over $f_1,\dots,
f_m$:
\begin{remark} Let $\sigma \in \{<, =, >\}^m$ and $E_{\sigma} = \{i
\mid \sigma_i = ``=" \}$. Due to Proposition
\ref{alcanza_con_de_a_n}, in the third step of Algorithm
\ref{algoritmo_general} it suffices to consider those sets $S\subset
\{1,\dots m\}$ such that $E_{\sigma} \subset S$. Then, if
$\#E_{\sigma} = l$, in the complexities of  Theorems
\ref{algoritmopuntosencomponentes} and \ref{allsignconditions}, the
sum can be taken over $l \le s \le \min \{m,n\}$ and the
combinatorial factor $\binom{m}{s}$ can be replaced by
$\binom{m-l}{s-l}$.
\end{remark}

\section{Closed sign conditions over arbitrary
polynomials}\label{sec:closed}

In the case of arbitrary polynomials, our approach is similar to the
one in \cite{BaPoRo96}. We will consider the same kind of
deformations as in the previous section but we will use different
initial systems whose particular properties enable us to recover
extremal points from the solutions of the deformed systems.

\subsection{Initial systems for deformations}

Let $d$ be an even positive integer and $T$ the Tchebychev
polynomial of degree $d$.

\begin{defn} \label{sistemaTcheby}
For a given $s\in \N$ with $1<s<n$ and $r:=s+n-1$, a \emph{type $2$
initial system} is a polynomial system of the form:
$$
\left\{\begin{array}{ll}g_i(x)=\tau_i\Big(n + A_{i(n+1)} + \sum_{1
\le k \le n}
A_{ik}T(x_k)\Big)   & \hbox{ for } 1\le i \le  s, \\[0.3cm]
g_i(x, \mu) = \displaystyle\sum_{1 \le j \le s} \mu_j\frac{\partial
g_{j}}{\partial x_{i-s+1}}(x)  & \hbox{ for } s+1\le i \le r,
\end{array}\right.
$$
where $\tau_i \in \{+, -\}$ for $1 \le i \le s$, and $A\in
\Q^{s\times (n+1)}$ is the Cauchy matrix defined as $A_{ik} =
\frac{1}{a_i +k}$ for $1 \le i \le s, 1 \le k \le n+1$, with $0 \le
a_1 < \dots <a_s$  integers such that $a_s + n + 1$ is a prime
number.

For $s=1$, a \emph{type $2$ initial system} consists of a polynomial
$g_1(x)$ as above and its partial derivatives $\frac{\partial
g_{1}(x)}{\partial x_{2}}, \dots, \frac{\partial g_{1}(x)}{\partial
x_{n}}$. Finally, for $s=n$, a \emph{type $2$ initial system}
consists of $n$ polynomials $g_1,\dots, g_n$ constructed as above
from the Cauchy matrix $A\in \Q^{ n\times (n+1)}$.

\end{defn}

Note that if $\tau_{i} = ``+"$, then $g_i(x)>0$ for every $x\in
\R^n$ and, if $\tau_{i} = ``-"$, then $g_i(x)<0$ for every $x\in
\R^n$. Moreover, for $s+1 \le i \le r$,
\begin{equation}\label{segunda_parte_sist_Tcheby}
g_i(x, \mu) = T'(x_{i-s+1})\Big(\sum_{1 \le j \le
s}\tau_jA_{j(i-s+1)}\mu_j\Big).\end{equation}

The B\'ezout number of a type $2$ initial system is $D =
\binom{n-1}{s-1} d^s(d-1)^{n-s} \le \binom{n-1}{s-1} d^n$.

\begin{lemma}\label{todo_bien_Tcheby}
Property (H) (see Section \ref{sistemainic}) holds for any type $2$
initial system.
\end{lemma}

\begin{proof}{Proof:} Assume $1 < s < n$. For $s+1 \le i \le r$, let $\bar g_i(x, \mu_1, \dots, \mu_{s-1}) =
g_i(x, \mu_1, \dots, \mu_{s-1}, 1)$.

Let $B \subset \{2, \dots, n\}$ be a set with  $ n-s$ elements, let
$e:B \to \{-1,1\}$ and suppose $e(k) = 1$ for  $a$ elements in $B$.
Let $S_{B,e}$ be the set of  solutions  $(\bar x, \bar \mu) = (\bar
x_1, \dots, \bar x_n, \bar \mu_1, \dots, \bar \mu_{s-1})$ of the
system
\begin{equation} \label{sist_Tch_desh}
g_1(x) = \dots =  g_s(x) = \bar g_{s+1}(x, \mu) = \bar g_{r}(x, \mu)
= 0
\end{equation}
which also satisfy
\begin{equation} \label{cond_Be}
T'(x_k) = 0 \hbox{ and } T(x_k) = e(k) \hbox{ for every } k \in B.
\end{equation}

We are going to compute the number of elements of $S_{B,e}$. Without
loss of generality, suppose $B = \{s+1, \dots, n\}$. As $\gcd(T',
T+1) = T_{ d/2}$ and  $\gcd(T', T-1) = T'/T_{ d/2}$ (where $T_{d/2}$
is the Tchebychev polynomial of degree $d/2$), the number of
$(n-s)$-uples $(\bar x_{s+1}, \dots, \bar x_{n})$ satisfying
(\ref{cond_Be}) is $(d/2)^{n-s-a}(d/2-1)^a$. Each of these elements
can be extended to solutions $(\bar x, \bar\mu)$ of
(\ref{sist_Tch_desh}) in $d^s$ different ways: Let $A' \in \Q^{s
\times s}$ be the Cauchy matrix formed by the first $s$ columns in
$A$. The conditions $g_1(x) = \dots = g_s(x) = 0$ and
(\ref{cond_Be}), imply
\begin{equation} \label{primera_parte_sis_Tcheby}
A'\left(\begin{array}{c} T(x_{1}) \\[0.5cm]
                                 \vdots \\[0.5cm]
                                 T(x_{s}) \\
                               \end{array}
                             \right) = - \left(\begin{array}{c} n + A_{1(n+1)} + \displaystyle\sum_{s+1 \le k \le n}A_{1k}e(k) \\
                                 \vdots \\
                              n +  A_{s(n+1)} + \displaystyle\sum_{s+1 \le k \le n}A_{sk}e(k) \\
                               \end{array}
                             \right).
\end{equation}
Using Cramer's rule for the previous system, we deduce that, for
each solution $(\bar x_1, \dots, \bar x_s)$,  $T(\bar x_{k_0}) \ne
\pm 1$  ($1 \le k_0 \le s$) because all these numbers are rational
with denominator a multiple of the prime number $a_s + n + 1$ and
numerator relatively prime to it. As the equation  $T(x) = \alpha$
only has multiple solutions for $\alpha = \pm 1$, for every $1 \le
k_0 \le s$, there are exactly $ d$ values $\bar x_{k_0}$ for which
$T(\bar x_{k_0})$ satisfy (\ref{primera_parte_sis_Tcheby}). It can
be easily shown that each solution  $(\bar x_1, \dots, \bar x_n)$
of (\ref{cond_Be}) and (\ref{primera_parte_sis_Tcheby}) can be
extended uniquely to a solution $(\bar x, \bar \mu)$ of
(\ref{sist_Tch_desh}).

As for $(\bar x, \bar \mu) \in S_{B,e}$, $T(\bar x_k) = e(k) = \pm
1$ for every $k \in B$, the sets $S_{B, e}$ are mutually disjoint.
Then, if $(\bar x^{(1)}, \bar \mu^{(1)}), (\bar x^{(2)}, \bar
\mu^{(2)})$ are two different solutions of (\ref{sist_Tch_desh}),
$\bar x^{(1)} \ne \bar x^{(2)}$. By taking into account every $B
\subset \{2, \dots, n\}$, every  $a$ $(0 \le a \le n-s)$  and every
function $e:B \to \{-1, 1\}$ whose value is  $1$ at exactly $a$
elements in $B$, we find
$$\binom{n-1}{n-s}\sum_{0 \le a \le n-s}\binom{n-s}{a}
\Big(\frac{d}{2}\Big)^{n-s-a} \Big(\frac{ d}{2} - 1\Big)^{a} d^{s} =
\binom{n-1}{n-s}\big( d - 1\big)^{n-s}  d^{s}$$ solutions to
(\ref{sist_Tch_desh}).

Consider now the Jacobian matrix of this system evaluated at each of
these solutions and suppose, without loss of generality, that the
solution $(\bar x, \bar \mu)$ considered corresponds to $B=\{s+1,
\dots, n\}$. Then, this matrix is of the form
$$ \begin{array}{cc}
\begin{array}{cc}
          s &\{ \cr s-1 & \{  \cr n-s &\{ \cr
         \end{array} & \left(\begin{array}{c|c|c}
                        \ C_1 & 0 & \ 0 \cr
                        \hline  *  & 0 &  \ C_2  \cr
                        \hline  *  &  \ C_3  \ &  \ *  \cr
                         \end{array}\right) \cr &
         \begin{array}{ccc}  \, \underbrace{}_s  & \underbrace{}_{n-s} \ &  \hspace{-2mm}\underbrace{}_{s-1}   \cr
                                     \end{array}
         \end{array}
 $$
 with $C_1$, $C_2$ and $C_3$ invertible matrices.

For $s=1$ and $s=n$, the proof is similar.
\end{proof}

\subsection{Geometric properties}

Let $f_{1} , \dots , f_{m} \in \R[x_1, \dots, x_n]$ and let $ S:= \{
i_1,\dots, i_s\} \subset \{1, \dots, m\}$ with $1<s<n$. As explained
in Section \ref{definicion_sistemas}, for every point $z\in\R^n$
with maximum or minimum first coordinate over the set $\{ x\in \R^n
\mid f_{i_1}(x) = 0,\dots, f_{i_s}(x)=0\}$, there exists a nonzero
vector $\mu = (\mu_1,\dots, \mu_s)$ such that $(z, \mu)$ is a
solution to the system
\begin{equation}\label{sistemaS}
\left\{\begin{array}{l}f_{i_1}(z)= \dots = f_{i_s}(z) = 0,
\\[1mm]
\sum_{1\le j \le s} \mu_j \frac{\partial f_{i_j}}{\partial x_2}(z) =
\dots = \sum_{1\le j \le s} \mu_j \frac{\partial f_{i_j}}{\partial
x_n}(z) = 0.\end{array}\right.
\end{equation}
Now, a homotopic deformation of this system by means of a type $2$
initial system is as follows: for every $1\le k \le s$, $F_k(t, x) =
(1-t)f_{i_k} + tg_{k}(x),$ and, for every $s+1 \le k \le r$,
$$F_k(t, x, \mu) = (1-t)\sum_{1 \le j \le s} \mu_j\frac{\partial f_{i_j}}{\partial x_{k-s+1}}(x) + t
\sum_{1 \le j \le s} \mu_j\frac{\partial g_{j}}{\partial
x_{k-s+1}}(x) = \sum_{1 \le j \le s} \mu_j\frac{\partial
F_{j}}{\partial x_{k-s+1}}(t, x).$$ Thus, for any $t_0 \in \R$ and
every $x_0\in \R^n$ at which the first coordinate function attains a
local maximum or minimum over the set $\{x \in \R^n \mid F_{1}(t_0,
x) = \dots = F_{s}(t_0, x) = 0\}$, by the implicit function theorem,
there is a nonzero vector $\mu_0\in \R^s$ such that $F_1(t_0, x_0) =
\dots = F_s(t_0, x_0) = F_{s+1}(t_0, x_0, \mu_0) = \dots = F_s(t_0,
x_0, \mu_0) = 0$.

\bigskip

In the sequel we will consider  deformations by means of specific
type $2$ initial systems.

Let $d\in \N$ be an even positive integer with $d\ge \deg f_i$ for
every $1\le i \le m$. Let $q_1< \dots < q_m$ be the first $m$ prime
numbers greater than $n$. For $1 \le i \le m$, let
$$g_i^+(x)= n + \frac{1}{q_i} + \sum_{1 \le k \le n} \frac1{q_i - n - 1 + k}T(x_k) \quad \hbox{ and } \quad
g_i^-(x)= -g_i^+(x).$$ Note that for each $S = \{i_1, \dots, i_s\}
\subset \{1, \dots, m\}$ with $1 \le s \le n$ and every list
$\tau_1, \dots, \tau_s$ of  $+$ and $-$ signs, the polynomials
$g_{i_1}^{\tau_1}, \dots, g_{i_s}^{\tau_s}$ form a type $2$ initial
system with  $a_j = q_{i_j} - n - 1$ for $1 \le j \le s$ (see
Definition \ref{sistemaTcheby}). In addition,  for $1\le i \le m$,
we denote
$$F_i^+(t,x) = (1-t)f_i(x) + tg_i^+(x) \quad \hbox{ and } \quad F_i^-(t,x) = (1-t)f_i(x) + tg_i^-(x).$$

\begin{lemma}\label{alcanza_de_n_tcheby}
Let $S = \{i_1, \dots, i_s\} \subset \{1, \dots, m\}$ with $s
> n$ and $\tau_1, \dots, \tau_s$ a list of $+$ and $-$ signs. Then,
the set $\{t \in \C \mid \exists x \in \C^n \hbox{ with }
F_{i_1}^{\tau_1}(t, x) =  \dots = F_{i_s}^{\tau_s}(t, x) = 0\}$ is
finite (possibly empty).
\end{lemma}
\begin{proof}{Proof: } Denote by $\hat F_{i_1}^{\tau_1},  \dots,
\hat F_{i_s}^{\tau_s}, \hat g_{i_1}^{\tau_1}, \dots, \hat
g_{i_s}^{\tau_s}$
 the polynomials obtained by homogenizing  $F_{i_1}^{\tau_1}, \dots,
F_{i_s}^{\tau_s}$ and $g_{i_1}^{\tau_1}, \dots, g_{i_s}^{\tau_s}$
with a new variable $x_0$.

Let $Z\subset \A^1 \times \P^n$ be the set of common zeros of $\hat
F_{i_1}^{\tau_1}, \dots, \hat F_{i_s}^{\tau_s}$ and $\pi_t:\A^1
\times \P^n \to \A^1$ the projection. In order to prove the
statement it suffices to show that $\pi_t(Z)$ is a finite set. Since
$\pi_t$ is a closed map, this can be proved by showing that $1\notin
\pi_t(Z)$, or equivalently, that the system $\hat
g_{i_1}^{\tau_1}(x) =  \dots = \hat g_{i_s}^{\tau_s}(x) = 0$ has no
solution in $\P^n$.

First, note that, if $(1:x_1:\dots: x_n)$ is a solution to this
system, then $(T(x_1),\dots, T(x_n))$ is a solution to the linear
system $B. y^t = -(n+\frac{1}{q_{i_1}}, \dots,
n+\frac{1}{q_{i_{n+1}}})^t$, where $B\in \Q^{(n+1)\times n}$ is the
Cauchy matrix of coefficients of $g_{i_1}^{\tau_1},\dots,
g_{i_{n+1}}^{\tau_{n+1}}$. But this linear system has no solutions,
since its augmented matrix has a nonzero determinant (as in the
proof of Lemma \ref{todo_bien_Tcheby}, it can be shown that this
determinant is a rational number whose numerator is not a multiple
of $q_{i_{n+1}}$).

Finally,  we have that, for $1 \le j \le s$, $\hat
g_{i_j}^{\tau_j}(0, x_1, \dots, x_n) = \pm 2^{d - 1} \sum_{1 \le k
\le n} \frac{1}{q_{i_j} - n - 1 + k}x_k^{d}$. Considering the
equations for $1\le j \le n$, we deduce that $(x_1^{d}, \dots,
x_n^{d})$ is in the kernel of the Cauchy matrix $(\frac{1}{q_{i_j} -
n - 1 + k})_{1 \le j, k \le n}$ and therefore, it is the zero
vector. Thus, the system $\hat g_{i_1}^{\tau_1}(x) =  \dots = \hat
g_{i_s}^{\tau_s}(x) = 0$ has no solutions in $\{x_0 = 0\}$.
\end{proof}

\begin{notn}
For $S = \{i_1,\dots, i_s\} \subset \{1, \dots, m\}$ with  $1 \le s
\le n$, and  $\tau = (\tau_1, \dots, \tau_s)\in \{+, -\}^s$, we
denote $\hat V_{S, \tau}\subset \A^1\times \A^n \times \P^{s-1}$ the
variety defined by the polynomials constructed as in
(\ref{sistemadeformado}) by taking $h_1,\dots, h_r$ as the
polynomials in system (\ref{sistemaS}) and $g_1,\dots, g_r$  the
type $2$ initial system given by $g_{i_1}^{\tau_1},\dots,
g_{i_s}^{\tau_s}$. We consider the decomposition $\hat V_{S, \tau} =
V^{(0)}_{S, \tau}\cup V^{(1)}_{S, \tau}\cup V_{S, \tau}$ as in
(\ref{hatV}).
\end{notn}

The following proposition will enable us to adapt  Algorithm
\ref{algoritmo_general} in order to solve the problem in this
general setting.

\begin{proposition}\label{caso_general_extremales_buscados}
Let $\sigma \in \{\le, =, \ge\}^m$, $E_\sigma = \{i \mid \sigma_i =
``="\}$,
  $U_{\sigma} = \{i \mid \sigma_i = ``\ge"\}$ and $L_{\sigma} = \{i \mid \sigma_i = ``\le"\}$.
 For $S = \{i_1,\dots, i_s\} \subset \{1, \dots, m\}$ with $1 \le s \le
 n$, let
${\cal T}_S = \{\tau \in \{+, -\}^s \mid \tau_j = ``+" \hbox{ if }
i_j\in U_\sigma \hbox{ and } \tau_j = ``-" \hbox{ if } i_j \in
L_\sigma\}$. Then, after a generic linear change of variables, for
each connected component  $C$ of $\{x \in \R^n \ | \ f_1(x) \sigma_1
0, \dots, f_m(x) \sigma_m 0\}$, we have
$$Z(C)  \subset \bigcup_{{S \subset \{1, \dots, m\}}\atop{1 \le \# S \le n}} \ \bigcup_{\tau \in {\cal T}_S}
\pi_x\big(V_{S, \tau}\cap \{t=0\}\big).$$
\end{proposition}

\begin{proof}{Proof:} Without loss of generality, we may assume that
$E_\sigma = \{1, \dots, l\}$, $U_\sigma = \{l+1, \dots, k\}$, and $
L_\sigma = \{k+1, \dots, m\}$ for some $l$, $k$ with $0 \le l \le k
\le m$; that is, we consider a connected component $C$ of
$\mathcal{P}=\{f_1 = 0, \dots, f_l = 0,f_{l+1} \ge 0,\dots, f_k \ge
0, f_{k+1} \le 0, \dots, f_m \le 0\}$. By Proposition
\ref{cambio_var_gen_finit_extrem}, after a generic linear change of
variables,  $Z(C)$ is finite. Moreover, since $\mathcal{P}$ is a
closed set, $Z(C) \subset C$. Let $z \in Z_{\rm inf}(C)$ and $0<
\varepsilon <1$ such that:
\begin{itemize}
\item $\overline{B(z, \varepsilon)} \cap \mathcal{P} \subset C$ and $\overline{B(z, \varepsilon)} \cap Z(C) =\{z\}$,
\item for every $\hat S = \{i_1, \dots, i_{\hat s}\} \subset \{1, \dots, m\}$ with $\hat s  > n$ and every
$(\tau_1, \dots, \tau_{\hat s}) \in \{+,-\}^{\hat s}$, $\varepsilon
< |t_0|$ for every $t_0$ in
     $\{t \in \C\setminus \{0\} \mid \exists x \in \C^n \hbox{ with }
F_{i_1}^{\tau_1}(t, x) =  \dots = F_{i_{\hat s}}^{\tau_{\hat s}}(t,
x) = 0\}$,
\item for every $S \subset \{1, \dots, m\}$ with  $1 \le \#S \le n$ and every $\tau \in {\cal T}_S$,
$\varepsilon < |t_0|$ for every $t_0 \in \C$ such that $V_{S,
\tau}^{(1)}$ has an irreducible component contained in $\{t =
t_0\}$.
\end{itemize}

For $t \in \R$, let $R_t$ be the set
$$\displaylines{\{x \in \overline{B(z, \varepsilon)} \mid F_1^+(t, x) \ge 0, \dots, F_l^+(t, x) \ge 0, F_1^-(t, x) \le 0, \dots, F_l^-(t, x) \le
0,\cr F_{l+1}^+(t, x) \ge 0, \dots, F_{k}^+(t, x) \ge 0,
F_{k+1}^-(t, x) \le 0, \dots, F_{m}^-(t, x) \le 0\}.}$$ We have that
$R_0 = C \cap \overline{B(z, \varepsilon)}$ and, for every $t \in
[0,1]$, $z \in R_t$.

Let $\nu>0$ be the distance between the compact sets $\partial B(z,
\varepsilon) \cap \{x_1 \le z_1\}$ and $R_0$.

We claim that for some $t_1$, $0 < t_1 < \varepsilon$, the connected
component  $C'$ of $R_{t_1}$ containing $z$ is included in $\{x \in
\overline{B(z, \varepsilon)} \ | \ d(x, R_0) \le \nu/2\}$. Suppose
this is not the case. Let $(t'_n)_{n \in \N}$ be a decreasing
sequence of positive numbers converging to $0$ and with $t'_1 <
\varepsilon$, and for every  $n \in \N$, let $C'_n$ be the connected
component of $R_{t'_n}$ containing $z$. Since $C'_n$ intersects $\{x
\in \overline{B(z, \varepsilon)} \ | \ d(x, R_0) > \nu/2\}$, there
is a point $r_n \in C'_n$ such that $d(r_n, R_0) = \nu/2$. Consider
the sequence $(r_n)_{n\in \N}\subset \overline{B(z, \varepsilon)}$,
and a subsequence $(r_{n_j})_{j\in \N}$ convergent to an element $r
\in \overline{B(z, \varepsilon)}$; then, $d(r, R_0) = \nu/2$. Now,
for $1 \le i \le k$, we have that $F_i^+(t_{n_j}, r_{n_j}) \ge 0$
for every $j\in \N$, and so, $F_i^+(0, r) \ge 0$, and for every $1
\le i \le l$ and $k+1 \le i \le m$, $F_i^-(t_{n_j}, r_{n_j}) \le 0$,
which implies that $ F_i^-(0, r)  \le 0$. Therefore $r \in R_0$,
contradicting the fact that $d(r, R_0) = \nu/2 > 0$.

Let $w \in C'$ be a point at which the function $x_1$ attains its
minimum over $C'$. Since $z \in C'$, we have $w_1 \le z_1$. If $w
\in
\partial B(z, \varepsilon)$, then $w \in \partial B(z,
\varepsilon) \cap \{x_1 \le z_1 \}$, and so, $d(w, R_0) \ge \nu$,
contradicting the fact that $d(w, R_0) \le \nu/2$. Therefore, $w \in
B(z, \varepsilon)$.

As each of the polynomials $F_1^+, \dots, F_l^+, F_1^-, \dots,
F_l^-, F_{l+1}^+, \dots, F_{k}^+, F_{k+1}^-, \dots, F_{m}^-$ that
does not vanish at $(t_1,w)$ takes a constant sign in a neighborhood
of this point,  we conclude that, if $F_{i_1}^{\tau_1}, \dots,
F_{i_s}^{\tau_s}$ are all the polynomials vanishing at $(t_1, w)$,
then the function $x_1$ attains a local minimum  over the set $\{x
\in \R^n \ | \ F_{i_1}^{\tau_1}(t_1, x) = 0, \dots,
F_{i_s}^{\tau_s}(t_1, x) = 0 \}$ at $w$. Let $S_0 = \{i_1, \dots,
i_s\}$, which is not empty. For $1 \le i \le m$, $F_i^+(t_1, w)$ and
$F_i^-(t_1, w)$ cannot be both zero; then, $i_1, \dots, i_s$ are all
distinct. Because of the way we chose $\varepsilon$, we also have
that $s \le n$. Now, if $\tau_0=(\tau_1, \dots, \tau_s)$, we have
that $(t_1, w) \in \pi_{t,x}(\hat V_{S_0, \tau_0})$, but taking into
account that $0 < t_1 < \varepsilon$, it follows that $(t_1, w) \in
\pi_{t,x}(V_{S_0, \tau_0})$. Therefore, $(t_1, w) \in \bigcup_{{S
\subset \{1, \dots, m\}}\atop{1 \le \# S \le n}} \ \bigcup_{\tau \in
{\cal T}_S}
    \pi_{t,x}( V_{S, \tau})$ and $0< |(t_1, w) - (0, z)| <
\sqrt2\varepsilon$.

Since the previous construction can be done for every
$\varepsilon>0$ sufficiently small and the sets $\pi_{t,x}( V_{S,
\tau})$ are closed, we conclude that $(0, z) \in \bigcup_{{S \subset
\{1, \dots, m\}}\atop{1 \le \# S \le n}} \ \bigcup_{\tau \in {\cal
T}_S}\pi_{t,x}( V_{S, \tau}).$ Therefore,
$$z \in \bigcup_{{S \subset \{1, \dots, m\}}\atop{1 \le \# S \le n}} \ \bigcup_{\tau \in {\cal T}_S}
    \pi_{x}\big( V_{S, \tau} \cap \{t=0\}\big).$$
\end{proof}

\subsection{Symbolic deformation algorithm}

In the sequel, $\Omega$ will denote a positive real number such that
for any ring $R$, addition, multiplication and the computation of
determinant and adjoint of matrices in $R^{k \times k}$ can be
performed within $O(k^\Omega)$ operations in $R$. We may assume
$\Omega \le 4$ (see \cite{Berk}) and, in order to simplify
complexity estimations, we will also assume that $\Omega \ge 3$.

\begin{proposition}\label{complalgoritmopararesgeomtipo2}

There is a probabilistic algorithm that, taking as input polynomials
$h_1, \dots, h_{r}$ in $ \K [x_1, \dots, x_n, \mu_1, \dots, \mu_s]$
as in (\ref{sistema_final}) encoded by an slp of length $L$, obtains
a geometric resolution of a finite set containing $\pi_x(V \cap \{
t=0\})$ for a deformation defined from a type 2 initial system
within complexity $O\big( n^3(L +  dn + n^{\Omega-1}) D^2\log^2(
D)\log\log^2(D)\big)$, where $d$ is an even integer such that $d\ge
\deg_x(h_i)$ for every $1\le i \le r$.
\end{proposition}

\begin{proof}{Proof:} The structure of the algorithm is similar to
that of the algorithm underlying the proof of Proposition
\ref{algoritmopararesgeom}.

\medskip

\noindent \textsc{First step}:  Take $\alpha = (\alpha_1, \dots,
\alpha_n) \in \Q^n$ at random and compute a geometric resolution
associated to the linear form $\ell_\alpha(x)=\alpha_1 x_1 + \cdots
+ \alpha_n x_n$ of the variety defined in $\A^r$ by the
(dehomogenized) type 2 initial system.

As shown in the proof of Lemma \ref{todo_bien_Tcheby}, this variety
can be partitioned into subsets $S_{B, e}$. So, we first compute a
geometric resolution associated with $\ell_\alpha(x)$ for each
$S_{B, e}$:  after solving a linear system of the type
(\ref{primera_parte_sis_Tcheby}), the $x$-coordinates of points in
$S_{B, e}$ turn to be defined by a square polynomial system in
separated variables; then, the required computation can be achieved
as in \cite[Section 5.2.1]{JMSW} within complexity  $O( D_{B, e}^2
\log^2( D_{B, e})\log \log(D_{B, e}))$, where $D_{B, e}$ is the
cardinality of $S_{B, e}$.

Finally, a geometric resolution of the whole variety is obtained
following the splitting strategy given in \cite[Algorithm
10.3]{vzG}, and noticing that, if $\{q, q_0, w_1, \dots, w_r\}$ and
$\{\tilde q, \tilde q_0, \tilde
 w_1, \dots, \tilde  w_r\}$ are geometric resolutions of disjoint sets with $q$ and $\tilde q$ coprime polynomials,
then $\{q\tilde q, q_0\tilde q + \tilde q_0q, w_1\tilde q + \tilde
w_1q, \dots, w_r\tilde q + \tilde w_rq\}$ is a geometric resolution
of their union. This can be done within $O(n D \log^2 (D) \log \log
(D))$ operations in $\Q$.

The whole complexity of this step is $O(n D^2 \log^2 (D) \log \log
(D)).$

\ms

\noindent\textsc{Second step}: Compute $P(t, U, y) \mod ((t-1)^{2nD
+ 1} + (y_1 - \alpha_1, \dots, y_n - \alpha_n)^2)\K[[t-1]][U, y].$

First, from the geometric resolution computed in the previous step,
obtain a geometric resolution associated with
$\ell(x,y)=y_1x_1+\cdots+y_nx_n$ of the variety defined by the
initial system over $\overline{\K(y)}$, modulo the ideal $(y_1 -
\alpha_1, \dots, y_n - \alpha_n)^2$, applying \cite[Algorithm
1]{GLS01} within complexity $O(( dn^2 + n^\Omega) D^2 \log^2 (D)
\log \log^2 (D))$. Then, consider the variety defined by $F_1,
\dots, F_r$ over $\overline{\K(t, y)}$ (see Lemma
\ref{prop_var_t_como_coef}). Since $F_1, \dots, F_r$ can be encoded
by an slp of length $L + O(( d + s)n)$, a geometric resolution of
this variety associated with the linear form $\ell(x, y)$ modulo the
ideal $(t-1)^{2n D + 1} + (y_1 - \alpha_1, \dots, y_n - \alpha_n)^2$
can be obtained from the previously computed geometric resolution by
applying \cite[Algorithm 1]{GLS01} within complexity $O(n^3(L +  dn
+ n^{\Omega-1})  D^2 \log^2 ( D) \log \log^2 ( D))$.

\ms

\noindent \textsc{Third step}: From the approximation to $P(t,U,y)$
obtain the required geometric resolution, by performing the same
computations as in the third and forth steps of the algorithm
underlying the proof of Proposition \ref{algoritmopararesgeom},
which does not modify the overall complexity.
\end{proof}

The algorithm underlying the proof of Proposition
\ref{complalgoritmopararesgeomtipo2} can be adapted
straightforwardly to handle the cases $s=1$ and $s = n$ within the
same complexity bounds.

\subsection{Main algorithm}

Here we prove the main result of this section.

\begin{theorem} \label{algoritmopuntosencomponentescerradas}
Given polynomials $f_1, \dots, f_m \in \K[x_1, \dots, x_n]$ with
degrees bounded by an even integer $d$ and encoded by an slp of
length $L$, for generic choices of the parameters required at
intermediate steps, the algorithm obtained from Algorithm
\ref{algoritmo_general} taking $p = (0,\dots, 0)$ and replacing step
3.(b) with
\begin{enumerate}
\item[(b')] For every $\tau \in \{+,-\}^{\#S}$, compute a geometric
resolution $\{q^{(k, S, \tau)}(U),v_k^{(k, S, \tau)}(U), \dots,$ $
v_n^{(k, S, \tau)}(U)\} \subset \K[U]$ of a finite set containing
$\pi_x(V_{k,S, \tau} \cap \{t = 0\})) \subset \A^{n-k+1}$ by means
of a deformation from a type 2 initial system, and add the geometric
resolution
$$\Big\{q^{(k, S, \tau)}(U), 0, \dots, 0, v_k^{(k, S,
\tau)}(U), \dots, v_n^{(k, S, \tau)}(U)\Big\}$$ to the list ${\cal
G}$.
\end{enumerate}
computes a finite set ${\cal M} \subset \A^n$ intersecting each
connected component of the realization of every feasible closed sign
condition over $f_1, \dots, f_m$. The complexity of the algorithm is
$O\Big(n^6\big(L + d + n^2\big)\log^2( d)\big(\log(n) + \log\log(
d)\big)^2d^{2n}\big(\sum_{s =1}^{\min\{m, n\}}
2^s\binom{m}{s}\binom{n-1}{s-1}^2\big)\Big).$
\end{theorem}

\begin{proof}{Proof:}
As in the proof of Theorem \ref{algoritmopuntosencomponentes}, by
Proposition \ref{idea_algoritmo}, it suffices to show that
\linebreak $\bigcup_{1 \le k \le n-1} \ \bigcup_{C \in {\cal
C}(k,p)} Z(C,k) \subset {\cal M},$ which is a consequence of
Proposition \ref{caso_general_extremales_buscados}.

Taking into account the linear change of variables performed at the
first step of the algorithm, for $1 \le k \le n-1$ and every $S
\subset \{1, \dots, m\}$ with $1 \le \#S \le n - k + 1$, we can
obtain an slp of length $O(nL + n^3)$ encoding the polynomials
involved in system (\ref{sistemaS}). Then, the algorithm underlying
the proof of Proposition \ref{complalgoritmopararesgeomtipo2}
computes a geometric resolution of a finite set containing
$\pi_x(V_{k, S, \tau} \cap \{t = 0\})) \subset \A^{n-k+1}$. The
stated complexity is obtained by adding up the complexities of these
steps  for all $(k, S,\tau)$.
\end{proof}

Proceeding as in Theorem \ref{allsignconditions}, from a finite set
$\mathcal{M}$ intersecting each feasible sign condition over the
polynomials $f_1,\dots, f_m$, we can obtain the list of all closed
sign conditions over these polynomials by evaluating their signs at
the points in $ \mathcal{M}$. We deduce the following complexity
result:

\begin{theorem}
There is a probabilistic algorithm which, given polynomials $f_1,
\dots, f_m \in \K[x_1, \dots, x_n]$ of degrees bounded by an even
integer $d$  and encoded by an slp of length $L$, computes the list
of all feasible closed sign conditions over these polynomials within
$O \Big( \sum_{s=1}^{\min\{m, n\}} \binom{m}{s} \binom{n-1}{s-1}^{2}
d^{2n}\big( 2^s n^6 (L +d+n^2) \log^2( d) \big(\log(n) + \log\log(
d)\big)^2 +m  d^{{(\omega -2)}
n}\binom{n-1}{s-1}^{\omega-2}\big)\Big)$
 operations in $\K$.\hfill $\square$
\end{theorem}

\section{Some particular cases}

\subsection{The bivariate case}\label{m_polinomios_dos_variables}

Here we will show that when $n=2$, Algorithm \ref{algoritmo_general}
solves our main problem for an \emph{arbitrary} finite family of
polynomials.

\begin{lemma} \label{polysuder} \label{dospolis} Let $f \in \C[x_1, x_2]$ be a nonzero
polynomial with no factors in $\C[x_1]\setminus \{0\}$ and $f =
\prod_{1\le i \le a} p_i^{\delta_i}$ its irreducible factorization
in $\C[x_1, x_2]$. Let $g_1, g_2$ be polynomials satisfying
conditions (H), with $g_1$ relatively prime to $f$. Let $F_1 =
(1-t)f + t g_1$ and $F_2 =(1-t)\frac{\partial f}{\partial x_2} +t
g_2$ and $V$ be the variety defined in (\ref{hatV}).
 If $z \in
\C^2$ satisfies that either there is an index $i_0$ such that
$p_{i_0}(z) = \frac{\partial p_{i_0}}{\partial x_2 }(z) = 0$ or
there are indices $i_1 \ne i_2$ such that $p_{i_1}(z) = p_{i_2}(z) =
0$, then $z \in \pi_x(V\cap \{t=0\})$.
\end{lemma}

\begin{proof}{Proof:}
In order to simplify the notation, we write $f'= \frac{\partial
f}{\partial x_2}$. Set $I = (F_1, F_2)\subset \K[t,x_1,x_2]$. Then
$V \cap \{ t = 0\} = V( (I:t^\infty) + (t))$.

Let $h = \gcd(f, f')=\prod_{i}p_i^{\delta_i - 1}$, $h_1 =
f/h=\prod_{i}p_i$ and $h_2 = f'/h= \sum_{i}\delta_i\frac{\partial
p_i}{\partial x_2}(\prod_{j\ne i}p_j)$. We will  show that the
equalities $(I: t^\infty) = (I:t) = (F_1, F_2, h_2g_1 - h_1g_2)$
hold. Then, since each condition in the lemma implies that $h_1(z) =
h_2(z) = 0$, we deduce that $(0, z) \in V(I:t^\infty)$ and,
therefore, that $z \in \pi_x(V\cap \{t=0\})$.

To prove the second equality, first note that $(h_2g_1 - h_1g_2)t =
h_2 F_1 - h_1 F_2,$ which shows the inclusion $\supset$. Now, if
$p(t,x) \in (I:t)$, we have $p(x, t)t = (\alpha_1(t,x)t +
\alpha_0(x))F_1 + (\beta_1(t,x)t + \beta_0(x))F_2$ for polynomials
$\alpha_1, \alpha_0, \beta_1, \beta_0$. Substituting $t=0$, we
obtain $\alpha_0f = -\beta_0f'$ and, dividing by $h$, $\alpha_0h_1 =
-\beta_0h_2$. Then, there exists $c \in \C[x]$ such that $\alpha_0 =
ch_2$ and $\beta_0 = -ch_1$ and therefore, $p(t,x) = \alpha_1(t,x)
F_1 + \beta_1(t,x)F_2 + c(h_2g_1 - h_1g_2)$.

The first equality will be proved by showing that $(I:t^2) \subset
(I:t)$. Let $p(x, t) \in (I:t^2)$. Then, $p(x, t)t^2 = (a_2(x, t)t^2
+ a_1(x)t + a_0(x))F_1 + (b_2(x, t)t^2 + b_1(x)t + b_0(x))F_2$ for
certain polynomials $a_i, b_j$. Setting $t=0$ in this identity, it
follows that there is a polynomial $c_1 \in \C[x]$ such that $a_0 =
c_1h_2$ and $b_0 = -c_1h_1$. Now, by looking at the terms of degree
$1$ in $t$, we have that $a_0g_1 + b_0g_2 = -a_1f - b_1f'$. Then,
$c_1(h_2g_1 - h_1g_2) = h(-a_1h_1 - b_1h_2)$. We deduce that $h$
divides $c_1$: for every $i$ with $\delta_i>1$, we have that $p_i
\mid h_1$ and $p_i\nmid h_2 g_1$ (since $f$ and $g_1$ are relatively
prime); so, $p_i\nmid(h_2g_1 - h_1g_2)$ and, therefore
$p_i^{\delta_i - 1}\mid c$. If $c_2 \in \C[x]$ satisfies $c_1 =
c_2h$, we have $p(x, t)t = (a_2t + a_1)F_1 + (b_2t + b_1)F_2 +
c_2(f'g_1 - fg_2)$. Since $f'g_1-fg_2 = -(g_2 - f')F_1 + (g_1 -
f)F_2\in I$, we conclude that $p \in (I:t)$.
\end{proof}

\begin{proposition}\label{encuentraextremosdosvariables} Let $f_1,\dots, f_m$ be
arbitrary bivariate real polynomials and $\sigma\in \{<, = ,
>\}^m$. Then, after a generic linear change of variables, for each
connected component $C$ of $\{ x\in \R^n \mid f_1(x)\sigma_1 0 ,
\dots,f_m(x)\sigma_m 0 \}$ we have that $Z(C) \subset \bigcup_{S
\subset \{1, \dots, m\},\, 1\le \# S \le 2} \pi_x(V_{S}\cap
\{t=0\})$ where the varieties $V_S$ are defined from type 1 initial
systems.
\end{proposition}

\begin{proof}{Proof:} After a generic linear change of variables we
may assume that, for each connected component $C$,  either $\pi_1(C)
= \R$ or $Z(C)$ is a non-empty finite set (see Proposition
\ref{cambio_var_gen_finit_extrem}).

Assume $Z_{\inf}(C) \ne \emptyset$ and let $z=(z_1, z_2) \in
Z_{\inf}(C)$. Since $z\in
\partial C$,
there is an index $i_0$ such that $f_{i_0} \ne 0$ and $f_{i_0}(z) =
0$. If $f_{i_0}$ has two or more non-associate irreducible factors
in $\C[x]$ vanishing at $z$, or an irreducible factor vanishing at
$z$ and whose derivative with respect to $x_2$ also vanishes at $z$,
by Lemma \ref{dospolis}, $ z \in \pi_x(V_{\{i_0\}}\cap \{ t=0\})$
(note that, because of the generic change of variables, $f_{i_0}$
does not have factors of the form $x_i -\alpha$). Otherwise, there
is a unique irreducible factor $p$ of $f_{i_0}$ vanishing at $z$
which must have all real coefficients (since its complex conjugate
also divides $f_{i_0}$ and vanishes at $z$) such that
$\frac{\partial p}{\partial x_2} (z) \ne 0$.

By the implicit function theorem applied to $p$ at the point $z$,
there is a continuous curve $(x_1, x_2(x_1))$ defined in a
neighborhood of $z_1$, and a neighborhood of $z$ such that the
polynomial $p$ (as well as any power of $p$ and also $f_{i_0}$) has
constant signs above, below and on the curve in this neighborhood.
Since $z_1 = \inf \pi_1(C)$, there must be an index $i_1 \ne i_0$
such that $f_{i_1}(z) = 0$. Moreover, we may assume that  $f_{i_1}$
has a unique irreducible factor $q$ vanishing at $z$ that is not an
associate to $p$. In this second case, $z$ is an isolated point of
$W_{\{i_0, i_1\}}= V(f_{i_0}, f_{i_1})$ and then, by Lemma
\ref{finitud}, $ z \in \pi_x(V_{\{i_0, i_1\}}\cap \{ t = 0 \})$.
\end{proof}

Using this proposition, following the proof of Theorem
\ref{algoritmopuntosencomponentes}, we have:

\begin{theorem}\label{teoremadosvariables}
Algorithm \ref{algoritmo_general} is a probabilistic procedure that,
given polynomials $f_1, \dots, f_m \in \K[x_1, x_2]$ of degrees
bounded by $d\ge 2$ that are encoded by an slp of length $L$,
  obtains a finite set $\mathcal{M}$ intersecting the closure of
each connected component of the realization of every sign condition
on the polynomials within complexity
$O(m^2d^4\log(d)\log\log(d)(L+\log^2(d) \log\log(d)))$. \hfill
$\square$
\end{theorem}

\subsection{A single polynomial}\label{un_polinomio_n_variables}

In this section we will show the procedure described in Section
\ref{sec:closed} can be adapted to solve the problem of computing a
point in the closure of each connected component of $\{ f = 0\}$,
$\{ f> 0 \}$ and $\{f<0\}$ for an arbitrary polynomial $f \in
\K[x_1,\dots, x_n]$.

Let $ d $ be an even positive integer such that $d \ge \deg (f)$.
Let $T$ be the Tchebychev polynomial of degree $d$. We define $g(x)
= n+\frac1q +\sum_{k =1}^n \frac{1}{q-n-1+k} T(x_k) $, where $q$ is
the smallest prime greater than $n$; $F(t,x) = (1-t) f(x) + t g(x)$,
and, for $2 \le i \le n$, $F_i(t,x) = (1-t) \frac{\partial
f}{\partial x_i} (x)+ t \frac{\partial g}{\partial x_i}(x) =
\frac{\partial F}{\partial x_i}(t,x)$.

Note that $g > 0$ over $\R^n$ and, therefore, $F >0$ over the set
$\{f=0\}$. Moreover, this is a deformation with a type $2$ initial
system (see Definition \ref{sistemaTcheby}). As in Section
\ref{sistemainic}, we let $\hat V = \{F = F_2 = \dots = F_n =0 \}
\subset \A^1 \times \A^n$, and consider its decomposition
(\ref{sistemadeformado}).

\begin{lemma}\label{comphiperconigual} After a generic linear change of variables,
for each connected component $C$ of  $\{f=0\}$,  $\{f>0\}$ or
$\{f<0\}$, we have $Z(C) \subset \pi_x(V \cap \{t=0\})$.
\end{lemma}

\begin{proof}{Proof:}  Consider first a connected component $C$ of $\{f=0\}$
(for a similar approach in this case with an alternative
deformation, see \cite{RRS}). Let $z \in Z_{\inf}(C)$. Take
$\varepsilon
> 0$ such that $\overline{B(z, \varepsilon)}$ meets neither a
connected component of $\{f=0\}$ different from $C$ nor the finite
set $Z(C) \setminus \{z\}$, and $\varepsilon < |t_0|$ for each $t_0
\in \C$ such that $V^{(1)}$ has a connected component contained in
$\{t = t_0\}$.

Let $\mu \in (z_1, z_1 + \varepsilon)$ be such that $\partial B(z,
\varepsilon) \cap C \subset \{x_1 > \mu\}$. Without loss of
generality, we may assume that $f$ is positive over the compact set
$\partial B(z, \varepsilon) \cap \{x_1 \le \mu\}$. Let
$\varepsilon_0 \in (0, \varepsilon)$ be such that $F$ is positive
over $[-\varepsilon_0, \varepsilon_0] \times (\partial B(z,
\varepsilon) \cap \{x_1 \le \mu\})$. Now, let  $y \in B(z,
\varepsilon)$ with $y_1 < z_1$ (thus, $f(y) \ne 0$). Since
$F(-\varepsilon_0, z) < 0$, $F(\varepsilon_0, z)> 0$ and $F(0, y)
\ne 0$, there is a point $(t_1, \tilde y)$ in the union of the line
segments joining $(-\varepsilon_0, z)$, $(0, y)$, and $(0,y)$,
$(\varepsilon_0, z)$ respectively, such that $F(t_1, \tilde y) = 0$.
We have $t_1 \ne 0$ and $\tilde y_1 < z_1$. Let $ w\in \{x \in
\overline{B(z, \varepsilon)} \mid F(t_1,x) = 0\}$ be a point at
which the coordinate function $x_1$ attains its minimum over this
compact set. Note that $w\not \in
\partial B(z, \varepsilon)$, since $w_1<\tilde y_1 < z_1 <\mu$ and $F$
is positive over $[-\varepsilon_0, \varepsilon_0] \times (\partial
B(z, \varepsilon) \cap \{x_1 \le \mu\})$. Therefore, $w \in B(z,
\varepsilon)$ and so, $(t_1,w) \in \hat V$. Moreover, as $0< | t_1|
<\varepsilon$, we have $(t_1,w) \in V$. Since $0 < |(t_1, w) - (0,
z)| < \sqrt2\varepsilon$, and the construction can be done for an
arbitrary sufficiently small $\varepsilon>0$, it follows that $(0,
z) \in V$.

Assume now that $C$ is a connected component of $\{f>0\}$ and let $z
\in Z_{\inf}(C)$ (then, $f(z) = 0$). Let $\tilde C$ be the connected
component
 of $\{f=0\}$ containing $z$, and $\varepsilon
> 0$ such that $\overline{B(z, \varepsilon)}$ meets neither a
connected component of $\{f=0\}$ different from $\tilde C$ nor the
finite set $Z(C) \setminus \{z\}$, and $\varepsilon < |t_0|$ for
every $t_0 \in \C$ such that $V^{(1)}$ has an irreducible component
included in $\{t = t_0\}$.

Let $\mu \in (z_1, z_1 + \varepsilon)$ be such that $\partial B(z,
\varepsilon) \cap \overline{C} \subset \{x_1 > \mu\}$ and
$\gamma:[0,1] \to \R^n$ a continuous semialgebraic curve such that
$\gamma(0) = z$ and $\gamma((0,1]) \subset C \cap B(z, \varepsilon)
\cap \{x_1 < \mu\}$. Let $C_1$ be the connected component of $C \cap
B(z, \varepsilon)$ with $\gamma((0,1]) \subset C_1$. Take $t_1 \in
(-\varepsilon, 0)$ small enough so that $F(t_1, \gamma(1))>0$. Since
$F(t_1, \gamma(0)) < 0$, there exists $u \in (0, 1)$ such that
$F(t_1, \gamma(u)) =0$. Let $C'$ be the connected component of $\{x
\in B(z, \varepsilon) \ | \ F(t_1, x) = 0\}$ containing $\gamma(u)$.
As $\gamma(u) \in C' \cap C_1$, $C' \cup C_1$ is a connected set.
Therefore  $C' \subset C_1$, as $C' \subset B(z, \varepsilon)
\setminus \tilde C$ and $C_1$ is a connected component of this set.
Now let $K = C' \cup (\overline{B(z, \varepsilon)} \cap \{x_1 \ge
\mu\})$, which is a compact set, since $\overline{C'} = C' \cup
(\partial B(z, \varepsilon)\cap \overline{C'}) \subset C' \cup
(\partial B(z, \varepsilon) \cap \overline{C} ) \subset K$. If $w
\in K$ is a point at which the function $x_1$ attains its minimum
over $K$, then $w \not \in \overline{B(z, \varepsilon)} \cap \{x_1
\ge \mu\}$. Then, $w$ is a minimum of $x_1$ over the set $C' \cap
B(z, \varepsilon) \cap \{x_1 < \mu\}$. Therefore, $(t_1,w) \in \hat
V$, but since $0 < |t_1| <\varepsilon$, we have $(t_1,w) \in V$. As
before, we conclude that $(0, z) \in V$.
\end{proof}

Using this lemma, due to Proposition \ref{idea_algoritmo} and
proceeding as in the proof of Theorem
\ref{algoritmopuntosencomponentescerradas}, we have:

\begin{theorem}
\label{algoritmopuntosencomponentes_casohipersuperf} Given a
polynomial $f \in \K[x_1, \dots,x_n]$ of degree bounded by an even
integer $d$ and encoded by an slp of length $L$, for generic choices
of the parameters required at intermediate steps, the algorithm
obtained from Algorithm \ref{algoritmo_general} modifying step 3.(b)
in order to deal with type $2$ initial systems, computes a finite
set ${\cal M} \subset \A^n$ intersecting the closure of each
connected component of the sets $\{f<0\}$, $\{f=0\}$ and $\{f>0\}$.
The complexity of the algorithm is $O\big(n^5\big(L +  n d +
n^{\Omega - 1} \big)\log^2(d)\big(\log(n) + \log\log( d)\big)^2
d^{2n}\big).$ \hfill $\square$
\end{theorem}

\bibliographystyle{abbrv}

\begin{thebibliography}{10}

\bibitem{ABRW96}
M.-E. Alonso, E.~Becker, M.-F. Roy, and T.~W{\"o}rmann.
\newblock Zeros, multiplicities, and idempotents for zero-dimensional systems.
\newblock In {\em Algorithms in algebraic geometry and applications}, volume
  143 of {\em Progr. Math.}, pages 1--15. Birkh\"auser, Basel, 1996.

\bibitem{AuRoSED02}
P.~Aubry, F.~Rouillier, and M.~Safey El~Din.
\newblock Real solving for positive dimensional systems.
\newblock {\em J. Symbolic Comput.}, 34(6):543--560, 2002.

\bibitem{BGHM97}
B.~Bank, M.~Giusti, J.~Heintz, and G.~M. Mbakop.
\newblock Polar varieties, real equation solving, and data structures: the
  hypersurface case.
\newblock {\em J. Complexity}, 13(1):5--27, 1997.

\bibitem{BGHM01}
B.~Bank, M.~Giusti, J.~Heintz, and G.~M. Mbakop.
\newblock Polar varieties and efficient real elimination.
\newblock {\em Math. Z.}, 238(1):115--144, 2001.

\bibitem{BGHP04}
B.~Bank, M.~Giusti, J.~Heintz, and L.~M. Pardo.
\newblock Generalized polar varieties and an efficient real elimination
  procedure.
\newblock {\em Kybernetika (Prague)}, 40(5):519--550, 2004.

\bibitem{BGHP05}
B.~Bank, M.~Giusti, J.~Heintz, and L.~M. Pardo.
\newblock Generalized polar varieties: geometry and algorithms.
\newblock {\em J. Complexity}, 21(4):377--412, 2005.

\bibitem{BaPoRo96}
S.~Basu, R.~Pollack, and M.-F. Roy.
\newblock On the combinatorial and algebraic complexity of quantifier
  elimination.
\newblock {\em J. ACM}, 43(6):1002--1045, 1996.

\bibitem{BPR93}
S.~Basu, R.~Pollack, and M.-F. Roy.
\newblock A new algorithm to find a point in every cell defined by a family of
  polynomials.
\newblock In {\em Quantifier elimination and cylindrical algebraic
  decomposition (Linz, 1993)}, Texts Monogr. Symbol. Comput., pages 341--350.
  Springer, Vienna, 1998.

\bibitem{BPR}
S.~Basu, R.~Pollack, and M.-F. Roy.
\newblock {\em Algorithms in real algebraic geometry}, volume~10 of {\em
  Algorithms and Computation in Mathematics}.
\newblock Springer-Verlag, Berlin, 2003.

\bibitem{BS83}
W.~Baur and V.~Strassen.
\newblock The complexity of partial derivatives.
\newblock {\em Theoret. Comput. Sci.}, 22(3):317--330, 1983.

\bibitem{Berk}
S. Berkowitz.
\newblock On computing the determinant in small parallel time using a small
  number of processors.
\newblock {\em Inform. Process. Lett.}, 18(3):147--150, 1984.


\bibitem{BP94}
D.~Bini and V.~Y. Pan.
\newblock {\em Polynomial and matrix computations. {V}ol. 1. Fundamental
  algorithms}.
\newblock Progress in Theoretical Computer Science. Birkh\"auser Boston Inc.,
  Boston, MA, 1994.

\bibitem{BCR}
J.~Bochnak, M.~Coste, and M.-F. Roy.
\newblock {\em Real algebraic geometry}, volume~36 of {\em Ergebnisse der
  Mathematik und ihrer Grenzgebiete (3) [Results in Mathematics and Related
  Areas (3)]}.
\newblock Springer-Verlag, Berlin, 1998.

\bibitem{Burgisser}
P.~B{\"u}rgisser, M.~Clausen, and M.~A. Shokrollahi.
\newblock {\em Algebraic complexity theory}, volume 315 of {\em Grundlehren der
  Mathematischen Wissenschaften [Fundamental Principles of Mathematical
  Sciences]}.
\newblock Springer-Verlag, Berlin, 1997.

\bibitem{Canny93}
J.~Canny.
\newblock Improved algorithms for sign determination and existential quantifier
  elimination.
\newblock {\em Comput. J.}, 36(5):409--418, 1993.

\bibitem{Collins75}
G.~E. Collins.
\newblock Quantifier elimination for real closed fields by cylindrical
  algebraic decomposition.
\newblock In {\em Automata theory and formal languages (Second GI Conf.,
  Kaiserslautern, 1975)}, pages 134--183. Lecture Notes in Comput. Sci., Vol.
  33. Springer, Berlin, 1975.

\bibitem{CW}
D. Coppersmith and S. Winograd.
\newblock Matrix multiplication via arithmetic progressions.
\newblock {\em J. Symbolic Comput.}, 9(3):251--280, 1990.

\bibitem{GHHM+}
M.~Giusti, J.~Heintz, K.~H{\"a}gele, J.~E. Morais, L.~M. Pardo, and
J.~L.
  Monta{\~n}a.
\newblock Lower bounds for {D}iophantine approximations.
\newblock {\em J. Pure Appl. Algebra}, 117/118:277--317, 1997.


\bibitem{GHMMP98}
M.~Giusti, J.~Heintz, J.~E. Morais, J.~Morgenstern, and L.~M. Pardo.
\newblock Straight-line programs in geometric elimination theory.
\newblock {\em J. Pure Appl. Algebra}, 124(1-3):101--146, 1998.

\bibitem{GLS01}
M.~Giusti, G.~Lecerf, and B.~Salvy.
\newblock A {G}r\"obner free alternative for polynomial system solving.
\newblock {\em J. Complexity}, 17(1):154--211, 2001.

\bibitem{GriVo88}
D.~Y. Grigor'ev and N.~N. Vorobjov, Jr.
\newblock Counting connected components of a semialgebraic set in
  subexponential time.
\newblock {\em Comput. Complexity}, 2(2):133--186, 1992.

\bibitem{Heintz83}
J.~Heintz,
\newblock Definability and fast quantifier elimination in
algebraically closed fields.
\newblock \emph{Theoret. Comput. Sci.}
24 (1983), no. 3, 239--277.

\bibitem{HJSS}
J.~Heintz, G.~Jeronimo, J.~Sabia, and P.~Solern\'o.
\newblock Intersection theory and deformation algorithms: the multi-homogeneous
  case.
\newblock Manuscript.

\bibitem{HKPSW}
J.~Heintz, T.~Krick, S.~Puddu, J.~Sabia, and A.~Waissbein.
\newblock Deformation techniques for efficient polynomial equation solving.
\newblock {\em J. Complexity}, 16(1):70--109, 2000.


\bibitem{HeRoSo90}
J.~Heintz, M.-F. Roy, and P.~Solern{\'o}.
\newblock Sur la complexit\'e du principe de {T}arski-{S}eidenberg.
\newblock {\em Bull. Soc. Math. France}, 118(1):101--126, 1990.

\bibitem{HS82}
J.~Heintz and C.-P. Schnorr.
\newblock Testing polynomials which are easy to compute.
\newblock In {\em Logic and algorithmic (Zurich, 1980)}, volume~30 of {\em
  Monograph. Enseign. Math.}, pages 237--254. Univ. Gen\`eve, Geneva, 1982.

\bibitem{JMSW}
G.~Jeronimo, G.~Matera, P.~Solern\'o, and A.~Waissbein.
\newblock Deformation techniques for sparse systems.
\newblock To appear in {\em Found. Comput. Math.}


\bibitem{Koenig}
J.~K\"onig.
\newblock {\em {Einleitung in die allgemeine Theorie der algebraischen
  Gr\"o{\ss}en.}}
\newblock {Leipzig: B. G. Teubner. X u. 552 S. $8^{\circ}$ }, 1903.

\bibitem{Kronecker}
L.~Kronecker.
\newblock {Grundz\"uge einer arithmetischen Theorie der algebraischen
  Gr\"ossen. Festschrift.}
\newblock 1882.

\bibitem{Lecerf03}
G.~Lecerf.
\newblock Computing the equidimensional decomposition of an algebraic closed
  set by means of lifting fibers.
\newblock {\em J. Complexity}, 19(4):564--596, 2003.

\bibitem{Ped04}
P. Pedregal.
\newblock {\em Introduction to optimization}, volumen 46 de {\em Texts in
  Applied Mathematics}.
\newblock Springer-Verlag, New York, 2004.

\bibitem{Renegar92}
J.~Renegar.
\newblock On the computational complexity and geometry of the first-order
  theory of the reals. {I, II, III}.
\newblock {\em J. Symbolic Comput.}, 13(3):255--352, 1992.

\bibitem{Rouillier99}
F.~Rouillier.
\newblock Solving zero-dimensional systems through the rational univariate
  representation.
\newblock {\em Appl. Algebra Engrg. Comm. Comput.}, 9(5):433--461, 1999.



\bibitem{RRS}
F.~Rouillier, M.-F. Roy, and M.~Safey El~Din.
\newblock Finding at least one point in each connected component of a real
  algebraic set defined by a single equation.
\newblock {\em J. Complexity}, 16(4):716--750, 2000.

\bibitem{Safey08}
M.~Safey El~Din.
\newblock Testing sign conditions on a multivariate polynomial and
  applications.
\newblock To appear in Mathematics in Computer Science.

\bibitem{SafSch03}
M.~Safey El~Din and {\'E}.~Schost.
\newblock Polar varieties and computation of one point in each connected
  component of a smooth algebraic set.
\newblock In {\em Proceedings of the 2003 International Symposium on Symbolic
  and Algebraic Computation}, pages 224--231 (electronic), New York, 2003. ACM.

\bibitem{SafTre}
M.~Safey El~Din and P.~Tr\'ebuchet.
\newblock Strong bi-homogeneous B\'ezout theorem and its use in effective real
  algebraic geometry.
\newblock INRIA Research Report RR-6001, 2006.

\bibitem{Schost03}
{\'E}.~Schost.
\newblock Computing parametric geometric resolutions.
\newblock {\em Appl. Algebra Engrg. Comm. Comput.}, 13(5):349--393, 2003.

\bibitem{Seidenberg54}
A.~Seidenberg.
\newblock A new decision method for elementary algebra.
\newblock {\em Ann. of Math. (2)}, 60:365--374, 1954.

\bibitem{Shafarevich}
I.~R. Shafarevich.
\newblock {\em Basic algebraic geometry}.
\newblock Springer-Verlag, Berlin, study edition, 1977.

\bibitem{Tarski51}
A.~Tarski.
\newblock {\em A decision method for elementary algebra and geometry}.
\newblock University of California Press, Berkeley and Los Angeles, Calif.,
  1951.
\newblock 2nd ed.

\bibitem{Voisin}
C.~Voisin.
\newblock {\em Hodge theory and complex algebraic geometry. {II}}, volume~77 of
  {\em Cambridge Studies in Advanced Mathematics}.
\newblock Cambridge University Press, Cambridge, 2003.


\bibitem{vzg86}
J.~von~zur Gathen.
\newblock Parallel arithmetic computations: a survey.
\newblock In {\em Mathematical foundations of computer science, 1986
  (Bratislava, 1986)}, volume 233 of {\em Lecture Notes in Comput. Sci.}, pages
  93--112. Springer, Berlin, 1986.

\bibitem{vzG}
J.~von~zur Gathen and J.~Gerhard.
\newblock {\em Modern computer algebra}.
\newblock Cambridge University Press, New York, 1999.

\end{thebibliography}

\end{document}